\documentclass[aap,preprint]{imsart}
\usepackage{latexsym,amssymb,amsmath,amsfonts,amsthm,xcolor, microtype}
\usepackage{comment,enumitem}
\usepackage{comment}
\usepackage{MnSymbol}
\usepackage{hyperref}

\usepackage[mathlines]{lineno}

\usepackage{tikz}

\startlocaldefs

\newtheorem{theorem}{Theorem}
\newtheorem{definition}{Definition}
\newtheorem{assumption}{Assumption}
\newtheorem{lemma}{Lemma}
\newtheorem{remark}{Remark}
\newtheorem{prop}{Proposition}

\newtheorem{example}{Example} 
\newtheorem{question}{Question} 
\newtheorem{problem}{Problem} 

\def\eps{\varepsilon}

\def\beq{ \begin{equation} }
\def\eeq{ \end{equation} }
\def\mn{\medskip\noindent}

\def\ep{\epsilon}

\def\square{\vcenter{\vbox{\hrule height .4pt
			\hbox{\vrule width .4pt height 5pt \kern 5pt
				\vrule width .4pt} \hrule height .4pt}}}

\def\RR{\mathbb{R}}
\def\ZZ{\mathbb{Z}}

\def\clearp{}

\def\P{{\mathbb P}}     
\def\E{{\mathbb E}}     
\def\<{{\langle}} 
\def\>{{\rangle}} 
\def\1{{\bf 1}}         
\endlocaldefs

\begin{document}

\begin{frontmatter}
	
	\title{Stochastic PDEs on graphs as scaling limits of discrete interacting systems}
	\runtitle{SPDE on graphs}

	\begin{aug}
		
		
\author{Wai-Tong Louis Fan}\footnote{
	Department of Mathematics, Indiana University, Bloomington.}\footnote{Center of Mathematical Sciences and Applications, Harvard University, Cambridge.}
		

	\end{aug}

	\begin{abstract}
	Stochastic partial differential equations (SPDE) on graphs were recently introduced by Cerrai and Freidlin \cite{MR3634278}.
	This class of stochastic equations in infinite dimensions provides a minimal framework for the study of the effective dynamics of much  more complex systems. 
	However, how they emerge from microscopic individual-based models is still poorly understood, partly due to complications near vertex singularities. 		
	In this work, motivated by the study of the dynamics and the genealogies of expanding populations in spatially structured environments,
	we obtain a new class of SPDE on graphs of Wright-Fisher type which have nontrivial boundary conditions on the vertex set. We show that these SPDE arise as scaling limits of suitably defined biased voter models (BVM), which extends the scaling limits of Durrett and Fan \cite{MR3582808}.
	We further obtain a convergent simulation scheme for each of these SPDE in terms of a system of  It\^o SDEs, which is useful when the size of the BVM is too large for stochastic simulations. 
	These give the first rigorous connection between SPDE on graphs and more discrete models, specifically, interacting particle systems and interacting SDEs. 
	Uniform heat kernel estimates for symmetric random walks approximating diffusions on graphs are the keys to our proofs.
	Some open problems are provided as further motivations of our study.
		\end{abstract}
	
	
	
\end{frontmatter}

\section{Introduction}

Stochastic partial differential equations (SPDE) on graphs, more precisely SPDE whose spatial variables lie in a metric graph, first explicitly appear in Cerrai and Freidlin \cite{MR3634278, cerrai2019fast} as asymptotic limits of SPDE on two-dimensional domains that shrink to a graph. Here a graph $\Gamma$ is a continuous object consisting of all points on its edges, so the real line $\RR$ is a trivial example which has one edge and no vertex. 
These equations 
provide a ``{\it minimal}" framework for the study of the interplay between the solution of SPDE and the geometric properties of the underlying metric space, ``{\it minimal}" in the sense that the metric space is essentially one-dimensional yet flexible enough to incorporate nontrivial topologies and various boundary conditions on the vertex set. Such interplay between the evolution of the quantity (e.g. density of a population or concentration of a chemical) modeled by the equation and the spatial environment of the system
is of fundamental importance in scientific modeling and control. 
For example, an important problem in ecology is to
identify mechanisms that permit the coexistence of species in different geographical environments.
The role of space and stochasticity in shaping competition outcomes and biodiversity has been intensively studied in spatial evolutionary games. It has also been explored rigorously in the framework of interacting particle systems  (a.k.a.~stochastic cellular automata), as in Durrett \cite{MR2521876}, Lanchier and Neuhauser \cite{MR2209349, MR2739344} and 
Lanchier \cite{MR3161301},
to name just a few. See the seminal articles of Durrett and Levin \cite{durrett1994importance, durrett1994stochastic} about the importance of space in modeling. 

\medskip

A practical motivation for our study of  SPDE on graphs is to provide a theoretical foundation for  previous experimental work \cite{endler2003propagation, inankur2015development} and on-going modeling work 
on co-infection spread of defective and normal viruses. Here co-infection means simultaneous infection by two or more different types of virus particles.
Instead of the traditional petri dish, designed micro-arrays of network structures are used as the container of host cells \cite{endler2003propagation,  inankur2015development}. Virus infections and co-infections are then systematically initiated and observed. 
The aim is
to {\it predict the propagation speed and the spatial patterns for viral co-infections in spatially structured populations} of biological host cells. Insights obtained from these laboratory studies are potentially applicable to more complex real life epidemic networks, which is important in controlling epidemic spread \cite{riley2015five}.

Quantitative imaging and analysis of viral infection provides
extensive spatial-temporal data for validation and refinement of models.$\,$
However, a reliable mathematical framework is still missing. Deterministic models like PDE on graphs fail to capture the dynamics of viral particles or genomes, because fluctuations of  propagating infection fronts are typically observed. A more reasonable macroscopic model is instead an SPDE on the graph formed by the host cell environment. The question, then, is to {\it deduce the ``correct" SPDE on the graph}, based on the local and spatial interactions between the viral particles and the host cells. For example, where does the observed noise come from; more specifically, what is the magnitude of the noise term in the SPDE in terms of microscopic rules? Can local interactions near a vertex singularity lead to the emergence of new terms in the SPDE? In co-infection spread, what population-level signatures reveal emergence of new levels of cooperation and conflict between the defective and normal viruses?
In an on-going work, the author is developing various stochastic spatial models, including individual-based models \cite{MR2153370} such as interacting particle systems  and systems of stochastic reaction-diffusion equations, to model the joint evolution of defective and normal viruses. This paper  provides the theoretical foundation for \cite{endler2003propagation, inankur2015development}, which is still missing in the mathematics literature. 

\medskip

Another broader motivation for this paper is related to our long term goal to understand the {\it  genealogies in expanding populations and the resulting patterns of genetic heterogeneity}. 
This is important because medical treatments in cancer or epidemics may fail due to drug resistance, if one does not have an accurate knowledge of the mutational types present.
The genetic forces at work in a growing cancer tumor or in an infection spread are very similar to those in a population expanding into a new geographical area, in which most of the advantageous mutation occur near the front. See Lee and Yin \cite{lee1996detection} and Edmonds et al. \cite{edmonds2004mutations}. 
Existing studies for genealogies in expanding populations mostly rely on computer simulations \cite{klopfstein2005fate} and nonrigorous arguments \cite{hallatschek2007genetic, hallatschek2008gene, hallatschek2010life, korolev2012selective, lehe2012rate, nullmeier2013coalescent}. The first rigorous analysis for this is perhaps in Durrett and Fan \cite{MR3582808}, which provides a precise description of the lineage dynamics in terms of a coupled SPDE  of Fisher-Kolmogorov-Petrovsky-Piscounov (FKPP) type. However, the spatial domain is restricted to $\RR$. A rigorous analysis for the genealogies in $\RR^d$ for dimensions $d \ge 2$ seems difficult. Even more, a solution theory for the stochastic FKPP is not yet available in two or higher dimensions. 
Therefore, besides the ``minimality" mentioned in the first paragraph, SPDE on metric graphs provide a natural setting for further analysis. See Section \ref{S:Open} for some concrete open problems.

\medskip

{\bf Main question and significance. } 
Even though SPDE on graphs and their deterministic counterpart arise naturally in scientific problems and discoveries,  they are rather unexplored. In the mathematics literature, a subset of these equations seem to first (and so far only) explicitly have appeared in the rather recent work \cite{MR3634278, cerrai2019fast}. 
This is partly because SPDE is still considered to be a rather new and technical modeling approach compared with deterministic models, but
a more important reason is that it is not yet clear 
{\bf how  these equations emerge from interactions in the microscopic scale}, especially interactions near vertex singularities. This fundamental question, our focus in this work as suggested by our title, needs to be carefully investigated in order to answer more specific questions such as those raised in the previous paragraph about epidemic spread. $\,$ 
With increasingly advanced technology, more and more experimental data
describe both cell-level and population-level behavior. Thus,
connecting continuum models with discrete models not only can facilitate model
validation at both scales, but also provides complementary
perspectives of the complex dynamics (such as tumor growth and virus spread) under study. Increasing recognition
of these benefits has stimulated efforts to connect discrete
and continuum models in a variety of biological and ecological contexts; see  \cite{durrett1994importance} for different modeling perspectives and \cite{durrett1994stochastic, codling2008random, MR2521876} for some reviews.

\medskip

{\bf Known results. }
In standard settings such as domains in $\RR^d$, there are rigorous approximation schemes of SPDE by microscopic particle models, such as Sturm \cite{MR1986838} and Kotelenez and Kurtz \cite{MR2550362}, where \cite{MR1986838} is for a stochastic heat equation with a multiplicative colored noise term and  \cite{MR2550362} is for SPDE of McKean-Vlasov type.
Particle representations of  SPDE  are obtained by Kurtz and coauthors in \cite{MR1705602, etheridge2014genealogical, crisan2016particle}. 
Cox, Durrett and  Perkins \cite{MR1756003} showed that the rescaled long range voter model in dimension $d\geq 2$ converges to a super-Brownian motion. Durrett,  Mytnik and Perkins \cite{MR2164042} showed convergence of multi-type contact processes to a pair of
super-Brownian motions interacting through their collision local times.
For genealogies of super-processes, see the snake process of Le Gall \cite{MR1207305}, the historical process of Evans and Perkins \cite{MR1615329} and the lookdown process of Donnelly and Kurtz \cite{MR1728556}. 
$\;$
Our two practical motivations led us to first focus on the  stochastic FKPP, the base case model for an expanding population density exhibiting noisy wavefront, which is of the form
\begin{equation}\label{fkppR}
	\partial_t u      \,= \alpha\,\Delta u +\beta\,u(1-u) + \sqrt{\gamma\,u(1-u)}\dot{W},
\end{equation}
where $\dot{W}$ is the space-time white noise, $\alpha>0$ is the diffusion coefficient representing the average dispersal distance of the individuals,  $\beta\geq 0$ accounts for an average increase  and the last term represents fluctuations during reproduction where $\gamma\geq 0$ parametrizes the variance.

Muller and Tribe  \cite{MR1346264} gave the first rigorous convergence result that stochastic FKPP on  $\RR$ can arise as scaling limit of long range biased voter models (BVM). In \cite{MR3582808}, we generalized this  by scrutinizing all possible scalings for which this type of connections between SPDE \eqref{fkppR} and BVM are valid. These BVM  are idealized individual-based models for  an expanding population on the 1-dimen rescaled integer lattice $L_n^{-1}\ZZ$ whose points are called \textit{demes}. 
In the $n$-th model in \cite{MR3582808}, there is one cell at each point of 
the lattice
\begin{equation*}\label{Lattice1}
(L_n^{-1}\ZZ) \times \{1, \ldots, M_n \},
\end{equation*}
whose cell-type is either 1 (cancer cell) or 0 (healthy cell). Each cell in deme $w \in L_n^{-1}\ZZ$ only interact with the $2M_n$ neighbors  in  demes $w-L_n^{-1}$ and $w + L_n^{-1}$. Type-0 cells reproduce at rate $2M_nr_n$, type-1 cells at a higher rate $2M_n(r_n+ \beta_n)$ due to higher fitness. When reproduction occurs the offspring replaces a neighbor chosen uniformly at random. 
In the terminology of evolutionary games, this is birth-death updating. It is shown in \cite[Theorem 1]{MR3582808} that under the scalings  
\begin{equation}\label{scaling1}
	\frac{r_nM_n}{L_n^2}\to \alpha,\;M_n\beta_n\to \frac{\beta}{2},\;\frac{L_n}{M_n}\to \frac{\gamma}{4\alpha},\,L_n\to\infty \text{ as }n\to\infty,
\end{equation}
the local fraction of type 1 converges to the solution of equation \eqref{fkppR} on $\RR$.  The scalings \eqref{scaling1} were explained in \cite[Section 1.2]{MR3582808} and will be compared with those in this paper in Remark \ref{Algo}.

Connections between models of different scales, offered by these types of scaling limit theorems and also Theorems \ref{T} and \ref{prop:SDE} in this paper, not only provide complementary insights into the underlying mechanisms of the complex dynamical system, but are also of fundamental importance for model selection and analysis.
For instance, the above convergence tells us that 
the variance of the noise is of order $L_n/M_n$ near the wavefront where $u$ is bounded away from 0 and 1, which is important in predicting the propagation speed \cite{MR2793860}.
See 
\cite{brunet2006phenomenological} for some behaviors that are expected to hold in general for the class of the stochastic pulled fronts with weak noise.

\medskip
{\bf More related work. } 
For SDE in infinite dimensions, which is an abstract framework containing SPDE on graphs or on manifolds \cite{MR1730556, MR1919608}, see \cite{MR3236753} for general background and \cite{MR2785545} a comparison between different theories.
Despite recent breakthroughs including \cite{MR3071506, MR3274562, MR3406823, caravenna2015universality}, a solution theory for the stochastic FKPP in the dimensions $d\geq 2$ is still open.

One might replace the white noise by a  colored noise as in \cite{DKR09, MR1986838} to smooth out spatial correlations. This standard approach indeed gives well-defined SPDE in higher dimensions, but the particle approximation would be less intuitive.
There is also a large body of literature 
about SPDE arising as the fluctuation limits of interacting particle systems, which we do not attempt to give a survey. We refer the reader to  \cite{MR3298474}  for the stochastic burgers equation and \cite{MR3529160} for reaction-diffusion equations, and the references therein.

\mn
{\bf Mathematical contributions. } Our main contributions are as follows.

\begin{enumerate}[leftmargin=*]	
	
	\item Our scaling limit theorems give {\bf the first rigorous  connections between SPDE on graphs and discrete models}.  
	See Remark \ref{Algo}.  $\,$	
	Theorem \ref{T} generalizes \cite[Theorem 1]{MR3582808} 
	to the graph setting, paying special attention to the new spatial heterogeneity and vertex singularities. A new stochastic FKPP on graph emerges from a suitably defined BVM (Section \ref{S:BVM}). 	 
	In Theorem \ref{prop:SDE}, we give a convergent simulation scheme for this SPDE in terms of a system of interacting It\^o SDEs. 
	This scheme
	is useful when the size of the BVM is too large for stochastic simulations.

	\item Besides having a different type of limit theorems compared to \cite{MR3634278, cerrai2019fast}, we consider more general diffusions on graphs $\Gamma$. In fact, we identify suitable conditions on $\Gamma$ for the study of SPDE on graphs and
point out directions for generalizations. See Remark \ref{Rk:Compare}.
Our SPDE have an extra  boundary condition at the set of vertices, and the coefficients are typically non-Lipschitz. Well-posedness of these SPDE is established via a new duality in Lemma \ref{L:WFdual}.

	\item An application of results like the conjunction of Theorems \ref{T} and \ref{prop:SDE} is as follows: 
	Given a complex system (such as cancer cell dynamics) with fine details, one (i) starts with an individual-based model which elucidate, at a more fundamental level, single-particle interactions, spatial component and stochasticity of the system, then (ii) deduces the macroscopic evolution of the particles that emerges, as in Remark \ref{Algo}, which might be an SPDE/PDE describing the evolution of the particle density, and (iii) simulates the SPDE/PDE which is robust against the size of the microscopic system, using the interacting SDEs or numerical methods. We summarize this lesson for our case as: 
	$$\text{BVM } \to \text{ SPDE }\to \text{ interacting SDEs} $$
	Benefits actually go both ways: an intuitive way to understand SPDE is through scaling limits of discrete approximating systems, similar to the way one interprets Brownian motion as scaling limits of random walks.
	
	\item  
	Besides vertex singularities,
	a technical challenge in the proofs is to obtain  {\it uniform estimates of the transition kernel} of random walks on a discretized version $\Gamma^n$  of $\Gamma$.
	For this we need to impose an assumption on $\Gamma$. The volume-doubling property and the Poincar\'e inequality in Assumption \ref{A:graph} are enough for this paper, and we point to further generalizations in Remark \ref{Rk:graph}. Uniform estimates for the random walks and the diffusions obtained in Theorems \ref{T:RWHK}-\ref{T:DiffusionHK} and also the local CLT are of crucial importance in analyzing regularity properties of SPDE on $\Gamma$ in general.
	
	
	\item  The scalings discovered in Theorem \ref{T} enable one to generalize, to the graph setting, scaling limit results for coupled SPDE such as  \cite[Theorem 4]{MR3582808}. This is a key step towards the study of interacting populations of more than one species. Broadly speaking, this paper points to directions for various generalizations, such as defining  SPDE on random graphs and on fractals, studying SPDE defined through Walsh diffusions instead of symmetric diffusions, extending scaling limits of contact processes in \cite{MR1346264} to the graph settings, etc. See Sections \ref{S:Generalize} and \ref{S:Open} for more generalizations and open questions.
\end{enumerate}

\mn
The paper is organized as follows: We give preliminaries in Sections \ref{S:Prelim} and \ref{S:BVM}, including  assumptions on the graph $\Gamma$, diffusions and SPDE and the construction of BVM, before stating our main results in Section \ref{S:main}. 
Uniform heat kernel estimates for random walks on the discretized graph $\Gamma^n$ and the local CLT are obtained in Section
\ref{S:DHK}.
Sections \ref{S:Pf1} and \ref{S:SDE} are proofs for Theorem \ref{T} and Theorem \ref{prop:SDE}. Some extensions of our method and some open problems are offered in
Section \ref{S:Generalize} and Section \ref{S:Open} respectively. Finally, the notions of weak solutions and mild solutions to SPDE on graphs are written down in the Appendix for completeness.

\clearp 
	
\section{Diffusions and SPDE on metric graphs}\label{S:Prelim}
A {\bf metric graph} $(\Gamma,\,d)$ is a topological graph $\Gamma=(V,\,E,\,\partial)$ endowed with a metric $d$, where $V$ and $E$ are disjoint countable sets and $\partial$ is a map. 
The elements of $V$ and $E$ are respectively the vertices and the directed edges of $\Gamma$. 
The combinatorial structure of the graph is described by the map $\partial:\,E \to V\times (V\cup\{\infty\})$ which sends every edge $e$ to
an ordered pair $(e_{-},\,e_{+})$. For simplicity, we identify $e$ with its image $\partial (e)=(e_-,e_+)$. We call $e_-$ and $e_+$  the {\it initial vertex} and the {\it terminal vertex} of $e$ (self-loops, i.e. $e_+=e_-$, are allowed).

The \textit{degree} of $v\in V$ is defined as
$deg(v):=|E(v)|=|E^+(v)|+|E^-(v)|$,
where $E^{\pm}(v):=\{e\in E:\,e_{\pm}=v\}$ consists of all edges starting ($-$) and ending (+) at $v$ respectively, and 
$E(v):=E^{+}(v)\bigcup  E^{-}(v)$ is their disjoint union. 
%

There are two types of edges.  If $(e_{-},\,e_{+}) \in V\times V$, then $e$ is called an \textit{internal edge}; if $(e_{-},\,e_{+}) \in V\times \{\infty\}$, then $e$ is called  an \textit{external edge}. 
Each internal edge is isomorphic to a closed and bounded interval in $\RR$, and each external edge is isomorphic to $[0,\infty)$.
We identify  $\Gamma$ with the union of all edges.
The metric $d$ on $\Gamma$ is defined in the canonical way as the length of a shortest path between two points along the edges. Internal edges have finite lengths and external edges have infinite lengths.
We equip $\Gamma$ with the 1-dimensional Hausdorff measure $m$ associated with $d$ and  write this metric measure space as $(\Gamma,d,m)$.

\medskip

Denote by $\mathring{e}$ the interior of the edge $e$  and $\mathring{\Gamma}=\Gamma\setminus V=\cup_{e\in E}\,\mathring{e}$ be the interior of $\Gamma$. 
For a function $f:\Gamma\to \RR$, we define 
\begin{itemize}
	\item  $\nabla f_{\pm}(x)$ to be the one sided derivative of $f$ at $x\in \mathring{e}$, along $e$ {\it towards} $e_{\pm}$, 
	\item $\nabla f_{e\pm}(v)$ be the one sided derivative of $f$ at $v=e_{\mp}\in V$, along  $e$ {\it towards} $e_{\pm}$
\end{itemize}
whenever they exist.
A function $f$ is said to be differentiable at $x\in \mathring{e}$ if $\nabla f_{+}(x)=-\nabla f_{-}(x)$, in which case this quantity $\nabla f(x)=\lim_{\mathring{e}\,\owns\, y\to x} \frac{f(y)-f(x)}{d(x,y)}$ is called the derivative of $f$ at $x\in \mathring{e}$. 
Higher derivatives $\nabla^kf$ (for $k\geq 2$) are defined similarly. 

Recall that an edge  $e\in E$ is isomorphic to a closed interval in $\mathbb{R}$.
The space of $k$-times continuously differentiable on $e$ is 
\[
\mathcal{C}^k(e):=\left\{ g\in  \mathcal{C}^k(\mathring{e}):\, \nabla^rg \text{ is uniformly continuous on bounded subsets of }\mathring{e} \text{ for all } r\leq k
\right\}.
\]
If $g\in \mathcal{C}^k(e)$, then $\nabla^rg$ continuously extends to the closed set $e$ for all $r\leq k$.
We also let 
\begin{itemize}
	\item $\mathcal{C}(\Gamma)$  be the space of continuous functions on $\Gamma$,
	\item $\mathcal{C}^k(\Gamma)\,:=\left\{ f\in  \mathcal{C}(\Gamma):\; f|_e \in \mathcal{C}^k(e)\; \text{for all }e\in E
	\right\}$ where $f|_e$ is the restriction of $f$ on $e$,
	\item $\mathcal{C}^{\infty}(\Gamma):=\cap_{k\geq 1}\mathcal{C}^k(\Gamma)$,
	\item $\mathcal{C}_c(\Gamma)$ be the space of continuous functions with compact support on $\Gamma$,
	\item $\mathcal{C}^k_c(\Gamma)\,:=\mathcal{C}^k(\Gamma)\cap \mathcal{C}_c(\Gamma)$,
	\item $\mathcal{C}^{\infty}_c(\Gamma):=\mathcal{C}^{\infty}(\Gamma)\cap \mathcal{C}_c(\Gamma)$.
\end{itemize}

For an arbitrary sigma-finite measure $\mu$ on $\Gamma$ (e.g., the Hausdorff measure $m$ or the measure $\ell(x)m(dx)$ in \eqref{Def:nu}), we define the $L^2$-norm of $f$ 
\[
\|f\|_{L^2(\mu)}:=\sqrt{\int_{\Gamma}|f(x)|^2\,\mu(dx)}=\sqrt{\sum_{e}\int_{e}|f(x)|^2\,\mu(dx)}
\]
and the Sobolev norm $\|f\|_{W^{1,2}(\mu)}:=\|f\|_{L^2(\mu)}+\|\nabla f\|_{L^2(\mu)}$.
The space $L^2(\Gamma, \mu)$ is the set of $\mu$-measurable functions $f$ with $\|f\|_{L^2(\mu)}<\infty$  and  the Sobolev space  
$$W^{1,2}(\Gamma,\mu):=\big\{f\in L^2(\Gamma,\mu): \,\nabla f\in L^2(\Gamma,\mu) \big\}.$$




Our focus is particle approximation to a class of parabolic SPDE on $\Gamma$. 
Our notion of solution to such an SPDE, detailed in the Appendix, involves a diffusion on $\Gamma$. This motivates us to construct such diffusions next, under assumptions on $\Gamma$ and the diffusion coefficients.

\medskip

\noindent
{\bf Diffusions on metric graphs via Dirichlet forms. } Although a nontrivial diffusion can be defined on very general $\Gamma$ including fractals \cite{MR966175} and random graphs \cite{MR3034461,MR3630284},   
we make the following assumption on $\Gamma$, which ensures certain regularity on the transition density of symmetric diffusions. 
See Remark \ref{Rk:graph}.
\begin{assumption}\label{A:graph}
	The metric graph $(\Gamma,d,m)$  has a positive infimum for the branch lengths and satisfies the followings.
	\begin{enumerate}
		\item[1.] (Volume-doubling) There is a constant $C_{VD} >0$ such that
		\begin{equation}\label{VD_dif}
			m(B(x,2r))\leq C_{VD}\,m(B(x,r)) 
		\end{equation}
		for all $x\in \Gamma$ and $r>0$ where $B(x,r):=\{y\in \Gamma:d(x,y)<r\}$ is a ball. 
		\item[2.] (Poincar\'e inequality) There is a constant $C_{PI}>0$ such that
		\begin{equation}\label{PI_dif}
			\int_{B(x,r)}|f(y)-\bar{f}_{B}|^2\,m(dy) \leq C_{PI} \,r^2\int_{B(x,2r)}|\nabla f(y)|^2\,m(dy)
		\end{equation}
		for all $f\in W^{1,2}(\Gamma,m)$, $x\in \Gamma$, $r>0$, where $\bar{f}_{B}:=m^{-1}(B)\int_{B}f\,dm$ is the average value of $f$ over $B=B(x,r)$. 
	\end{enumerate}
\end{assumption}


\begin{remark}\rm\label{Rk:graph}
	Any graph with finitely many edges  satisfies  Assumption \ref{A:graph} with $C_{VD}=2\,\sum_{v\in V} deg(v)$ and $C_{PI}=C^2\sum_{v\in V} deg(v)$, where $C$ is the same absolute constant in Theorem 2 in \cite[Section 5.8]{MR2597943} with $p=2$ and $n=1$. 
	Many infinite graphs also satisfy  Assumption \ref{A:graph}. These include any regular infinite lattice such as $\ZZ^n$ and any infinite regular tree with constant branch length. 
	As we shall explain  in Subsection \ref{SS:HK}, the conjunction of \eqref{VD_dif} and  \eqref{PI_dif} is {\it equivalent to} the existence of two-sided Gaussian bounds for the transition density of symmetric diffusions on $\Gamma$. 
	Assumption \ref{A:graph} can be significantly relaxed.
\end{remark}


We shall construct diffusions on graphs by the Dirichlet form method, under the following conditions on the diffusion coefficients and the symmetrizing measure.
\begin{assumption}\label{A:Coe_a}
	Suppose we are given two functions $\alpha,\,\ell \in 	W^{1,2}(\Gamma,m)$ that are strictly elliptic, i.e. bounded above and below by positive constants. 
\end{assumption}

We now define the measure $\nu$ on $\Gamma$ by
\begin{equation}\label{Def:nu}
	\nu(dx)=\ell(x)\,m(dx),
\end{equation}
which has full support and is locally finite, and we
consider the symmetric bilinear form 
\begin{align*}
	\mathcal{E}(f,g)&:=\int_{\Gamma} \alpha(x)\,\nabla f(x)\cdot\nabla g(x) \,\nu(dx)
\end{align*}
with domain $Dom (\mathcal{E})=W^{1,2}(\Gamma,\nu)\cap \mathcal{C}^1(\Gamma)$.
It can be checked that $(\mathcal{E},\,Dom (\mathcal{E}))$ is a Dirichlet form in $L^2(\Gamma,\,\nu)$ that possesses the local property. See for instance  \cite{MR2849840}. Furthermore, $(\mathcal{E},\,Dom (\mathcal{E}))$ is {\it regular} under Assumptions \ref{A:graph} and \ref{A:Coe_a}. This can be checked (for instance \cite[Proposition 4.1]{MR3034461}) by showing that the space $\mathcal{C}^{\infty}_{c}(\Gamma)$ of infinitely differentiable functions on $\Gamma$ with compact support is a core of the domain. 
Hence by \cite[Theorem 7.2.1]{MR2778606}, there is a $\nu$-symmetric diffusion $X=\{X_t\}_{t\geq 0}$ on the graph $\Gamma$ associated with this Dirichlet form. Henceforth we refer to $X$ as the {\bf $\mathcal{E}$-diffusion} and denote by $\big(\Omega,\,\mathcal{F},\,\{\P_x\}_{x\in \Gamma}\big)$ the filtered probability space on which $X$ is defined, where $\P_x$ is a probability measure on $\big(\Omega,\,\mathcal{F}\big)$ such that $\P_x(X_0=x)=1$.  Whenever $\P$ appears in an expression, it denotes the probability measure on the space on which the random variables involved in that expression are defined. 



\begin{remark}\rm \label{Rk:Compare}
	In \cite{MR1245308,MR3634278,cerrai2019fast} such a diffusion process is constructed  by specifying its generator under stronger assumptions: the set of vertices $V$ is finite, $\alpha=1$ is a constant function and $\ell$ is smooth in $\mathring{\Gamma}$.  We follow the notation in \cite{MR3634278,cerrai2019fast} for $\Gamma,\ell$ and $\nu$.
	Studying solutions to SPDE on graphs rely crucially on our understanding about certain  diffusion processes on graphs, which is a research topic by itself.  
	See \cite{MR1245308, MR1743769} for theoretical foundation of diffusions on finite graphs,  \cite{lejay2003simulating, MR2813404, MR3634968} for some interesting applications.
\end{remark}

As is known \cite{MR2778606, MR2849840, MR1214375}, the Dirichlet form method is more robust against irregularities of both the diffusion coefficients and the underlying metric space. The price to pay, however, is that many statements about the associated process $X$ are valid  apriori only for ``quasi-everywhere", that is, except for a set of capacity zero. 
Fortunately, most of these statements can be strengthened to be valid for ``all $x\in \Gamma$", provided that  we have extra knowledge about its transition density. 
By the usual $L^2$ method (see, e.g., \cite{carlen1987upper}),
$X_t$ admits a  transition density $p(t,x,y)$ with respect to its symmetrizing measure $\nu(dx)$. That is
\begin{equation}\label{Def:p}
	\P_x(X_t\in dy)=p(t,x,y) \nu(dy)
\end{equation} 
and $p(t,x,y)=p(t,y,x)$ for all $t\geq 0$, $\nu$-almost all $x,y\in \Gamma$. The previous `almost all' can further be strengthened to `all' because \eqref{VD_dif} and \eqref{PI_dif} imply
H\"older continuity for $p(t,x,y)$. See Theorem \ref{T:DiffusionHK} for the precise statement.


\medskip

\noindent
{\bf Gluing condition and generator. }
For any edge $e\in E$, it is known (for example \cite[Example 1.2.3]{MR2778606}) that 
every function $f\in W^{1,2}(e,m)$
coincides a.e. with an absolutely continuous function on $\mathring{e}$ having
derivative defined $m$-a.e. and lies in $L^2(e,m)$. In particular, under Assumption \ref{A:Coe_a}, the one-sided limits
$$\alpha_{e\pm}:=\lim_{x\to e_{\pm}} \alpha(x)\quad\text{and}\quad \ell_{e\pm}:=\lim_{x\to e_{\pm}} \ell(x) \quad\text{exist for all }e\in E.$$

As in \cite[Section 3]{MR1245308},
the $L^2(\Gamma,\,\nu)$ infinitesimal generator $\mathcal{L}$ of $X$ can be described as a second order differential operator
\begin{equation}\label{GeneratorL}
	\mathcal{L}f(x):= \frac{1}{2\ell(x)}\nabla \big(\ell(x)\,\alpha(x)\,\nabla f(x)\big), \quad x\in \mathring{\Gamma},
\end{equation}
endowed with the gluing conditions 
\begin{align}\label{glue}
	0\,=\big(\nabla_{out}f\cdot [\ell\alpha]\big)(v), \quad v\in V,
\end{align}
where the following {\it outward} derivative of  $f$ at $v$ is used:
\begin{align}\label{OutDiv0}
	\big(\nabla_{out}f\cdot [g]\big)(v)
	\,&:= \sum_{e\in E^+(v)}\nabla f_{e-}(v)\,g_{e+}\,+\, \sum_{e\in E^-(v)}\nabla f_{e+}(v)\,g_{e-}.
\end{align}
The gluing condition \eqref{glue} reduces to the Neumann boundary condition at vertices with degree 1.
From \eqref{glue} and integration by parts, we can check that $\mathcal{E}(f,g)=\<-\mathcal{L}f,\,g\>_{L^2(\nu)}$ for $f\in Dom_{L^2}(\mathcal{L})$ and $g\in Dom(\mathcal{E})$, where $Dom_{L^2}(\mathcal{L})$ is the domain of the generator $\mathcal{L}$. See  \cite[Appendix A.4]{MR2849840} and  \cite[Chapter 1]{MR1214375} for general relations among the generator $\mathcal{L}$, the Dirichlet form  and the semigroup $\{P_t\}_{t\geq 0}$ of $X$.

\medskip

\noindent
{\bf SPDE on graphs. }In this work we focus on the following stochastic FKPP type equation on $\Gamma$  with nontrivial boundary condition:
\begin{equation}\label{fkpp0}
	\left\{\begin{aligned}
		\partial_t u      &\,= \mathcal{L} \,u +\beta\,u(1-u) + \sqrt{\gamma\,u(1-u)}\dot{W}
		& &\quad\text{on }\overset{\circ}{\Gamma} \\
		\nabla_{out}u\cdot [\alpha] &\,= -\hat{\beta}\, u(1-u)  & &\quad \text{on }V,
	\end{aligned}\right.
\end{equation}
where $\dot{W}$ is the space-time white noise on $[0,\infty)\times \Gamma$, the operator $\nabla_{out}$ is defined in \eqref{OutDiv0}, the graph satisfies Assumption \ref{A:graph}, 
the functions $\alpha, \ell$ satisfy 
Assumption \ref{A:Coe_a}, and the functions $\beta,\gamma,\hat{\beta}$ satisfy the following assumption. 
\begin{assumption}\label{A:Coe_b}
	Let $\beta,\,\gamma:\,\Gamma\to [0,\infty)$ be non-negative bounded measurable functions on $\Gamma$ and $\hat{\beta}:V\to [0,\infty)$ be a non-negative bounded function on $V$.
\end{assumption}
Here we adopt Walsh's theory (named after Walsh's St. Flour notes \cite{MR876085}) and regard \eqref{fkpp0} as a shorthand for an integral equation. 
See Definition \ref{Def:WeakSol} in Appendix for this integral equation and the definition of a weak solution to \eqref{fkpp0}. 

Note that when $\hat{\beta}(v)$ is  positive, the negative sign in the boundary condition corresponds to creation of mass (growth of $u$) at $v\in V$. 
To see this, note that $\nabla_{out}u\cdot [\alpha](v)$ is a sum of derivatives of $u$ \textit{away} from the vertex $v$, and therefore towards the interior of the edges connecting to $v$. These \textit{inward} normal derivatives along the edges correspond to a growth rather than a decay of $u$.

While  Assumptions \ref{A:graph} and \ref{A:Coe_a} guarantee existence and basic properties of the diffusion process $X$ on the graph, Assumption \ref{A:Coe_b} will be needed for the weak uniqueness of \eqref{fkpp0}. The latter  will be established via duality (Lemma \ref{L:WFdual}).

\medskip

To have a cleaner description without losing much generality of the spatial heterogeneity across edges, we further restrict to piecewise constant functions whenever a discrete approximation is involved (namely, in Theorem \ref{T}, Theorem \ref{prop:SDE}, Lemma \ref{Invar} and Theorem \ref{T:RWHK}).
\begin{example}\label{Eg:fkpp}
	Suppose  $\alpha(x)=\alpha_e$, $\beta(x)=\beta_e$ and $\gamma(x)=\gamma_e$ whenever $x \in \mathring{e}$, where  $\{\alpha_e,\,\beta_e,\,\gamma_e\}$ are non-negative numbers that are uniformly bounded and $\inf_{e\in E}\alpha_e>0$. Suppose  $\ell=1$ constant. Then SPDE \eqref{fkpp0} reduces to
	\begin{equation}\label{fkpp1}
		\left\{\begin{aligned}
			\partial_t u      &\,= \alpha_e\,\Delta u +\beta_e\,u(1-u) + \sqrt{\gamma_e\,u(1-u)}\dot{W}
			& &\quad\text{on }\overset{\circ}{e} \\
			\nabla_{out}u\cdot [\alpha] &\,= -\hat{\beta}\, u(1-u)  & &\quad \text{on }V.
		\end{aligned}\right.
	\end{equation}
	In this case  $\alpha_{e+}=\alpha_{e-}:=\alpha_e$ and for $v\in V$,
	\begin{align}\label{OutDiv}
		\big(\nabla_{out}u\cdot [\alpha]\big)(v)
		\,&=\sum_{e\in E(v)} \Big( {\bf 1}_{e_-=v} \nabla u_{e+}(v)+{\bf 1}_{e_+=v}\nabla u_{e-}(v) \Big)\,\alpha_{e}. 
	\end{align}
\end{example}


\tikzset{
	big dot/.style={
		circle, inner sep=0pt, 
		minimum size=3mm, fill=black
	}
}

\begin{figure}
	\begin{center}
		\begin{tikzpicture}[scale=0.8]
		
		\node (O) at (0,4) {};
		
		\node[big dot] (V1) at (6,4) {};
		\node at (6,3.5) {\LARGE $v_1$};
		
		\node[big dot] (V2) at (12, 7) {};
		\node at (12,6.5) {\LARGE $v_2$};
		
		\draw [line width=1mm] (O) -- (V1);
		\draw [line width=0.8mm] (V1) -- (V2);
		\draw [line width=0.8mm] (V1) -- (10.5,1);
		\draw [line width=0.8mm] (V2) -- (14.4, 5.8);
		
		\node at (1.5,5.5) {\Large $M^e$};
		\node at (3.05,3.5) {\Large $\frac{1}{L^e}$};
		\node at (0.7, 3.5) {\LARGE $e$};
		\node at (9.2, 1.3) {\LARGE $\widetilde{e}$};
		\node at (10, 0.5) {};
		
		\draw [-] (2,4) -- (2,7);
		\draw [-] (2.7,4) -- (2.7,7);
		\draw [-] (3.4,4) -- (3.4,7);
		\draw [-] (4.1,4) -- (4.1,7);
		\draw [-] (4.8,4) -- (4.8,7);
		\draw [-] (5.5,4) -- (5.5,7);
		
		\draw node[fill=green, circle, inner sep=0pt, minimum size=2.5mm] at (5.5,4) {};
		\draw node[fill=green, circle, inner sep=0pt, minimum size=2.5mm] at (4.8,4) {};
		
		\draw node[fill=red, circle, inner sep=0pt, minimum size=3mm] at (5.5,5.5) {};
		\node at (5.9,5.75) {\Large $w$};
		
		\draw node[fill=red, circle, inner sep=0pt, minimum size=3mm] at (7.5,4) {};
		\node at (7.9,4) {\Large $z$};
		
		\draw [-] (7,4.5) -- (7,5.5);
		\draw [-] (8,5) -- (8,6);
		\draw [-] (9,5.5) -- (9,6.5);
		\draw [-] (10,6) -- (10,7);
		\draw [-] (11,6.5) -- (11,7.5);
		
		\draw node[fill=green, circle, inner sep=0pt, minimum size=2.5mm] at (7,4.5) {};
		
		\draw [-] (12.8,6.6) -- (12.8,8.1);
		\draw [-] (13.6,6.2) -- (13.6,7.7);
		
		\draw [-] (7.5,3) -- (7.5,4.5);
		\draw [-] (9,2) -- (9,3.5);
		
		\draw node[fill=green, circle, inner sep=0pt, minimum size=2.5mm] at (7.5,3) {};
		
		\draw [->,ultra thick, cyan] (5.5,5.5) to [out=0,in=135] (7,5);
		\draw [->,ultra thick, cyan] (5.5,5.5) to [out=180,in=0] (4.8,5.5);
		\draw [->,ultra thick, cyan] (5.5,5.5) to [out=225,in=180] (7.5,4);
		\end{tikzpicture}
		
		\caption{Lattice  $\Lambda_n:= \cup_{e}\Lambda^{e}_n$, together with 
			sites $w\in e^n$ and $z\in \tilde{e}^{\,n}$ such that $w\sim z$.}\label{Fig:Lattice}
	\end{center}
\end{figure}
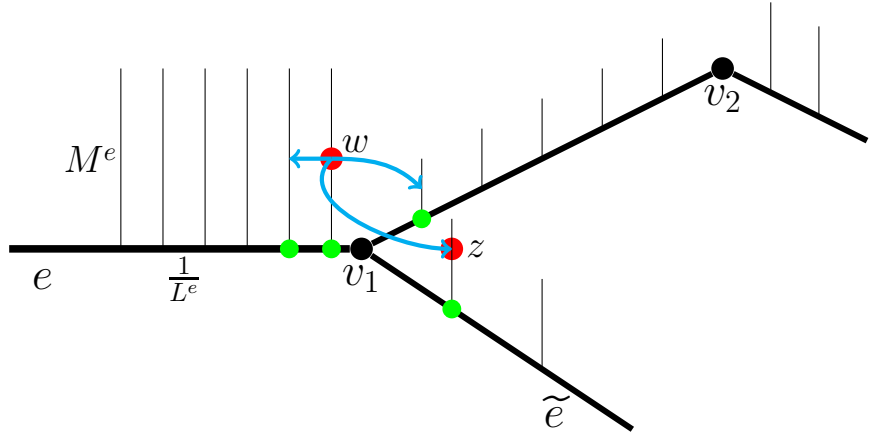	

\section{Rescaled biased voter model}\label{S:BVM}

In this section, we describe a sequence of biased voter models (BVM) indexed by  $n\in \mathbb{N}$, which is a natural generalization to the one in \cite{MR3582808} described in the introduction.

For each edge $e\in E$, we associate it with
two sequences $\{L^{e}_n\}_{n\geq 1}\subset (0,\infty)$ and $\{M^{e}_n\}_{n\geq 1}\subset \mathbb{N}$, then we choose a {\it maximal} countable subset $e^{n}$ of its interior $\mathring{e}$ such that 
neighboring points in $e^{n}$ are of distance $1/L^{e}_n$. Without loss of generality, by throwing away demes that are too close to the endpoints of an edge if necessary, we suppose
\begin{equation}\label{dxv}
	1/L^e_n \leq d(x,v) <2/L^e_n
\end{equation} 
whenever $x\in e^n$ is neighboring to vertex $v$.
Points in the  {\bf discretized graph} $\Gamma^{n}:=\cup_{e}e^{n}$ are called {\it demes}. Each deme $x\in e^n$ represents an isolated location containing a subpopulation of $M_n^e$ particles.  Two different demes $x, y\in \Gamma^{n}$ are said to be {\it neighbors} (denoted $x\sim y$) if either they lie on the same $e$ and $d(x,y)=1/L^{e}_n$ or if they are both adjacent to the same vertex $v\in V$. 

Our $n$-th BVM is defined on the {\bf discrete lattices}  ({\bf Figure \ref{Fig:Lattice}}) 
$$
\Lambda_n:= \cup_{e}\Lambda^{e}_n \qquad\text{where }
\Lambda^{e}_n:= \Big\{(x,\,i):\;x\in e^{n},\,i\in \{1, \cdots, M^{e}_n \} \Big\}.
$$
Points of $\Lambda_n$ are called {\it sites}. Each site contains one individual agent/particle,  which is of either type 1 or type 0.
In the context of cancer dynamics, we think of an agent/particle as a biological cell,  type 1 (cancer cell) and type 0 (normal cell). 
A site $z$ is said to be on edge $e$ (denoted $z\in e^n$) if $z\in \Lambda^{e}_n$; it is said to be in deme $x$ if $z=(x,i)$ for some $i$.
Two different sites $z,w\in \Lambda_n$ are said to be {\it neighbors} (also denoted $z\sim w$) if they are located at two neighboring demes. 

\medskip
\noindent
{\bf Dynamics of BVM. }Particles in deme $x$ only interact with those in neighboring demes.
Let $\xi_t(z):=\xi^{n}_t(z)$ be the type of the particle at site $z$ at time $t$. Our BVM  $(\xi_t)_{t\geq 0}$  can be constructed using  two independent families of  Poisson processes 
$\{P^{z,w}_t:\;z\sim w\}$ with rates $a^{z,w}$ and $\{ \tilde P^{z,w}_t:\;z\sim w\}$ with rates $b^{z,w}$.
At a jump time of $P^{z,w}_t$, the particle at $z$ is replaced by an offspring of the one at $w$. At a jump time of $\tilde P^{z,w}_t$, the particle at $z$ is replaced by an offspring of the one at $w$ {\bf only if $w$ has type 1}, so there is a ``bias" towards type 1. Under Assumption \eqref{A:graph}
the biased voter process $(\xi_t)_{t\geq 0}=(\xi^n_t)_{t\geq 0}$ is a Markov process with state space $\{0,1\}^{\Lambda_n}$.


\medskip

We suppress the superscript/lowerscript $n$'s in $\xi_t$, $a^{z,w}$, $b^{z,w}$, $L^e$ and $M^e$ to simplify notation. 

\section{Main results}\label{S:main}

The following assumption, to be explained in Remark \ref{Rk:RW} right after the main result, is crucial to both Theorems \ref{T} and \ref{prop:SDE}. 
For two different directed edges $e,\,\tilde{e}\in E(v)$ where $v\in V$, we let $L^{e,\tilde{e}}=(d(x,v)+d(v,y))^{-1}$  where $x\in e^n$ and $y\in \tilde{e}^n$ are two neighboring demes in $\Gamma^n$.  
	For $e=\tilde{e}$ we let  $L^{e,e}=L^e$.
\begin{assumption}[Symmetric conductances]\label{A:C}
	Suppose  $\{C^n_{\tilde{e},e}\}_{e,\tilde{e}\in E(v),\,v\in V}$ are positive numbers satisfying symmetry $C^n_{e,\tilde{e}}=C^n_{\tilde{e},e}$ and 
	\begin{equation}\label{Cond_C}
		\sup_{v\in V}\sup_{e,\tilde{e}\in E(v)} \,\Big|C^n_{\tilde{e},e}-L^{e,\tilde{e}}\,\Theta(\alpha_e,\,\alpha_{\tilde{e}})\Big|\,L^e\to 0, \qquad \text{as }n\to\infty,
	\end{equation}
	for some symmetric continuous function $\Theta$ on $(0,\infty)^2$ such that  $a\leq \Theta(a,b)=\Theta(b,a) \leq b$ for all $a\leq b$ and that $\Theta(a,a)=a$ for all $a\in (0,\infty)$. In particular, $\frac{C^n_{e,e}}{L^e}\to \alpha_e$. 
\end{assumption}
Examples include $\Theta(a,b)=\sqrt{ab}\,$ and power means $\left(\frac{a^p+b^p}{2}\right)^{1/p}$ with $p\in \RR\setminus \{0\}$.

\subsection{Scaling limit of BVM}

The principal result in this paper says that the approximate densities of our BVM converge to  SPDE \eqref{fkpp1} under suitable conditions, where the approximate density at deme $x\in  e^{n}$ is defined by
\begin{equation}\label{approx_den}
	u^n_t(x)  := \frac{1}{M^e} \sum_{i=1}^{M^e} \xi_t(x,i).
\end{equation}
For  $v\in V$, we define
$u^n_t(v)$ to be the average value of $\{u^n_t(x)\}$ among demes $x$ which are  adjacent to $v$. We then linearly interpolate between demes (and also between vertices and demes) to define $u^n_t(x)$ for all $x\in \Gamma$.
Then for all $t\geq 0$, we have $u^n_t\in \mathcal{C}_{[0,1]}(\Gamma)$,  the set of continuous functions on $\Gamma$ taking values in the interval $[0,1]$. Furthermore, if we equip $\mathcal{C}_{[0,1]}(\Gamma)$ with the metric 
\beq
\|\phi\| := \sum_{i=1}^{\infty} 2^{-i} \sup_{x\in K_i}|\phi(x)|
\label{normC}
\eeq
where $\{K_i\}$ is an increasing sequence of compact subsets of $\Gamma$ with the union of the sequence being $\Gamma$.
i.e., uniform convergence on compact sets, then $\mathcal{C}_{[0,1]}(\Gamma)$ is Polish and the paths $t\mapsto u^n_t$ are $\mathcal{C}_{[0,1]}(\Gamma)$-valued and c\`adl\`ag. 
Our main theorem is a weak convergence result for the sequence $\{u^n\}$ in the Skorohod space $D([0,\infty),\,\mathcal{C}_{[0,1]}(\Gamma))$.
\begin{theorem}\label{T}
	Let $(\Gamma,\,d)$ be a metric graph satisfying Assumption \ref{A:graph}. Let $\alpha,\beta$ and $\gamma$ be piecewise constant functions as in Example \ref{Eg:fkpp}, and  $\hat{\beta}:V\to [0,\infty)$ be a non-negative bounded function on $V$. 
	Suppose that as $n\to\infty$, the initial condition $u_0^n$ converges in $\mathcal{C}_{[0,1]}(\Gamma)$ to $f_0$, and that the following assumptions for $L^e,\,M^e,\,a^{zw}$ and $b^{zw}$ hold.
	\begin{enumerate}
		\item[(a)] $1/L^e \to 0$ uniformly for all $e\in E$ and \eqref{dxv} holds.
		\item[(b)] $4\,L^e/M^e \to \gamma_e/\alpha_e$ for all $e\in E$.
		\item[(c)]  $a^{z,w}=2\,C^n_{\tilde{e},e}L^e/M^e$ for all $z \in \tilde{e}^n,\, w \in e^n$  such that $z\sim w$, where
		$\{C^n_{\tilde{e},e}\}$ are positive numbers that satisfy Assumption \ref{A:C}.
		\item[(d)] 
		$b^{z,w}=B^n_e$ whenever $w\in e^n$ is not adjacent to a vertex in $V$ and  $z\sim w$, where $\{B^n_e\}$ are non-negative numbers such that $2\,B^n_e\,M^e\to \beta_e$.
		\item[(e)] For each $v\in V$, let $x^e_1$ be the element on $e^n$ adjacent to $v$ and $x^e_2\in e^n$ be adjacent to $x^e_1$. Suppose
		$\widehat{B}^n_{e,e}:=b^{x^e_2,\,x^e_1}$ and $\widehat{B}^n_{e,\tilde{e}}:=b^{x^e_1,\,x^{\widetilde{e}}_1}$ for $\tilde{e}\neq e$ in $E(v)$ are non-negative numbers and
		\begin{equation}\label{Cond_betav}
			4\,\sum_{e\in E(v)}\sum_{\tilde{e}\in{E(v)}}\frac{\widehat{B}^n_{e,\tilde{e}}\,M^{\tilde{e}}}{L^e} \to \hat{\beta}(v).
		\end{equation}
	\end{enumerate}
	Then 
	the  processes $(u_t^n)_{t\geq 0}$ in \eqref{approx_den} converge, in distribution in $D([0,\infty),\,\mathcal{C}_{[0,1]}(\Gamma))$, to a continuous $\mathcal{C}_{[0,1]}(\Gamma)$-valued process $(u_t)_{t\geq 0}$ which is the weak solution to the stochastic partial differential equation \eqref{fkpp1}
	with  $u_0=f_0$.
\end{theorem}

We explain  this result  in the remarks below.

\begin{remark}[Identifying SPDE from microscopic rules]\rm\label{Algo}
	The significance of Theorem \ref{T} lies in the connection it establishes between the microscopic BVM  and the new macroscopic SPDE model that have fewer parameters. For example, the BVM is intractable to analyze or simulate when $L^e$ or $M^e$ is large, but with  Theorem \ref{T} one can  take advantage of a new regularity (described by the SPDE, which is robust against the size of the particle system) that emerges. 
	
	
	To compare these two models we must relate microscale and macroscale parameters.
	The micro-parameters are $\{L^{e,\tilde{e}}\}$, $\{M^e\}$, $\{a^{z,w}\}$ and $\{b^{z,w}\}$. The macro-parameters can be found from micro-parameters as follows. 
	\begin{enumerate}
		\item[(i)] $\quad\alpha_e = 2\lim_{n\to\infty}C^n_{e,e}/L^e$, 
		\item[(ii)] $\quad\gamma_e=4\lim_{n\to\infty}\alpha_eL^e/M^e$, 
		\item[(iii)] 
		$\quad\beta_e= 2\lim_{n\to\infty}B^n_e \,M^e$, 
		\item[(iv)] $\hat{\beta}(v)= 4\lim_{n\to\infty} \sum_{e\in E(v)}\sum_{\tilde{e}\in{E(v)}}\widehat{B}^n_{e,\tilde{e}}\,M^{\tilde{e}}/L^e$.
	\end{enumerate}
	These are generalizations of \eqref{scaling1}.
	Conversely, given $\alpha,\beta,\hat{\beta}$ and $\gamma$ satisfying hypothesis of Theorem \ref{T}, there exist micro-parameters such that (i)-(iv) hold. 
	
	Using this connection, one can either obtain macro-parameters from microscopic (e.g. cell-level) measurements and experimental set up, or test hypothesis of microscopic interactions by using population-level measurements,  or even perform model validation at both scales.
	Different edges can have different $L^e_n$ or $M^e_n$, allowing the flexibility to model situations in which cells of different types and experimental configurations are situated on different edges \cite{su2013effect}. This also enables 
	us to take care of the case when a solution is simultaneously deterministic ($\gamma_e=0$) on edge $e$ and noisy ($\gamma_{\tilde{e}}>0$) on another edge $\tilde{e}$. 
\end{remark}

\begin{remark}[Random walk approximations]\rm \label{Rk:RW}
	The numbers $\{C^n_{e\tilde{e}}\}$ in Condition (c) of Theorem \ref{T}  arise naturally as the symmetric conductances of a random walk $X^n$ which, under \eqref{Cond_C}, converges in distribution to the $m$-symmetric diffusion $X$  with $\ell=1$ and $a=\alpha_e$ on $\mathring{e}$. More precisely, 
	define the measure $m_n$ on $\Gamma^n$ by
	$m_n(x):=\frac{1}{L^e}$ whenever $x\in e^n$.
	Let $X^n:=\{X_t^n\}_{t\geq 0}$ be the continuous time random walk (CTRW) on $\Gamma^n:=\cup_e e^n$ associated with the Dirichlet form $\mathcal E^n(f,g)$ on $L^2(\Gamma^n,m_n)$, where 
	\begin{equation}\label{RW_DF}
		\mathcal E^n(f,g):=\frac{1}{2}\sum_{x,y\in \Gamma^n}(f(x)-f(y))(g(x)-g(y))\, C^n_{xy} \qquad\text{and}
	\end{equation} 
	\begin{equation}\label{RW_sym}
		C^n_{xy}:=
		\begin{cases}
				C^n_{e,\tilde{e}} \quad \text{if }x\in e^n, \,y\in\tilde{e}^n \text{ and }x\sim y\\
				0\qquad  \text{ if }x \text{ and }y \text{ are not adjacent in }\Gamma^n.		
		\end{cases}
	\end{equation}

	Observe that $X^n$ is $m_n$-symmetric  since $C^n_{xy}=C^n_{yx}$ for all $x,y\in \Gamma^n$. 
	Under \eqref{Cond_C}, we  have $\mathcal E^n(f,g)\to \mathcal E(f,g)$ for $f,g\in \mathcal{C}^{\infty}_{c}(\Gamma)$.
	With extra work one can establish
	the weak convergence $X^n\to X$, as precisely stated in Lemma 
	\ref{Invar}. 
\end{remark}

\begin{remark}[Uniform approximation]\rm \label{eps_n}
	Condition (a) further implies a local central limit theorem (local CLT) and a uniform Holder continuity for $X^n$. The latter results,  established in Theorems \ref{T:LCLT_CTRW} and \ref{T:DiffusionHK} in Section \ref{S:DHK}, will be used to obtain tightness of $\{u^n\}$ (Proposition \ref{T:Tight}).
	When $1/L^e \to 0$ {\it uniformly} for all $e\in E$ [for instance when Condition (a) is in force], we shall  fix an edge $e*\,\in E$ and take $\eps_n:=1/L^{e*}$, a representative rate at which every $1/L^e$ tends to zero. 
	All $\{C^n_{x,y}\}_{x\sim y}$ are then of the same order $O(1/\eps_n)$ as $n\to\infty$ by \eqref{Cond_C}. 
\end{remark}

\begin{remark}[Generator $\mathcal{L}_{n}$]\rm \label{Rk:RWGen}
	The transition rate of $X^n$ from $x$ to $y$ is 
	$$\lambda(x,y)=C^n_{xy}/m_n(x):=C^n_{e\tilde{e}}\,L^e \quad \text{if }x\in e_n,\,y\in \tilde{e}_n,\,y\sim x.$$  
	Condition \eqref{Cond_C} implies that $|\lambda(x,y)- \alpha_e (L^e)^2|\to 0$ uniformly for $x,y\in e_n$ with $y\sim x$ and for all $e$. Hence the generator $\mathcal{L}_{n}$ of $X^n_t$ can be approximated by
	\begin{equation}\label{Generator_n1}
		\mathcal{L}_{n}F(x) \approx \alpha_e\Delta_{L^e}F(x)
	\end{equation} 	
	whenever $x\in e^n\setminus \{x^e_1\}$,
	where $\Delta_{L^e}$ is the discrete Laplacian in \eqref{dLaplacian}, and
	\begin{equation}\label{Generator_n2}
		\mathcal{L}_{n}F(x^e_1) \approx\,\alpha_e(L^e)^2\big(F(x^{e}_2) -F(x^e_1)\big)\,+\,\sum_{\tilde{e}\in E(v):\,\tilde{e}\neq e}\big(F(x^{\tilde{e}}_1) -F(x^e_1)\big)\,C_{e,\tilde{e}}\,L^e
	\end{equation}
	whenever $x^e_1$ is the element in $e_n$ which is adjacent to a vertex  and $x^e_2$  is the element in $e_n$ which is adjacent to $x^e_1$. The approximations $\approx$ in \eqref{Generator_n1} and \eqref{Generator_n2} can be quantified by using Condition \eqref{Cond_C}: the absolute difference between the left and the right is at most
	$o(\eps_n)\,\|F\|_{\infty}$ where  $o(\eps_n)$ represents a term independent of $F$  and which tends to 0 {\it uniformly} for all $x\in \Gamma^n$ faster than $\eps_n$. 
\end{remark}

\begin{remark}[Local growth at $v\in V$]\rm \label{Rk:Cond_betav}
	Results here for the simpler case $\hat{\beta}=0$ (no extra birth on $V$) are already new.
	Condition \eqref{Cond_betav} is crucially needed in (and only in) \eqref{Cond_betav2}. 
	It implies that, in order to have nontrivial boundary conditions $\hat{\beta}\neq 0$, the bias rates $b^{z,w}$ near vertices need to be of order at least $L^e/M^{\tilde{e}}$ which is typically higher than those in the interior of the edges. For example, $\hat{\beta}=0$ if all $\{b^{z,w}\}$ are of the same order in $n$ and so do all $\{M^e\}$.
	To try to give further interpretation, we  suppose for simplicity that $L^e$ are the same and that $\gamma_e>0$ for all $e\in E$. Condition (b) and \eqref{Cond_betav} roughly say that 
	\begin{equation*}
		\hat{\beta}(v)\approx
		4\sum_{e\in E(v)}\sum_{\tilde{e}\in{E(v)}}\widehat{B}^n_{e,\tilde{e}}\frac{\alpha_{\tilde{e}}}{{\gamma_{\tilde{e}}}}=
		4\sum_{\tilde{e}\in{E(v)}} \Big(\sum_{e\in E(v)}\widehat{B}^n_{e,\tilde{e}}\Big)\frac{\alpha_{\tilde{e}}}{{\gamma_{\tilde{e}}}},
	\end{equation*} 
	where $\sum_{e\in E(v)}\widehat{B}^n_{e,\tilde{e}}$ can be interpreted as a local growth at $v$ contributed by $\tilde{e}$. 
\end{remark}

\begin{remark}\rm
	When $\Gamma=\RR$, Theorem \ref{T} reduces to \cite[Theorem 1]{MR3582808}. Under the same scalings, we can obtain the corresponding generalization of \cite[Theorem 4]{MR3582808}. 
\end{remark}

\subsection{Interacting SDE as a numerical scheme}

Theorem \ref{T} enables us to derive an SPDE which captures the macroscopic evolution of the density of particles in the BVM. See Remark \ref{Algo}. So when the number of particles are too large, one can simulate the more robust SPDE instead of the stochastic particle system.  The next question is then: {\it how to simulate the SPDE?} $\,$
It is known  \cite{MR948717, MR1397705, MR1705602, MR2014157, MR2550362, MR3274891} that SPDE can arise as the continuum scaling limit of interacting SDEs. 
These SDEs provide a numerical scheme for the solutions of the SPDE and also the foundation for stochastic simulation algorithms (Gillespie algorithms \cite{gillespie1976general, gillespie1977exact}). 


In this section, we construct a system of interacting SDEs that offer a semi-discrete approximation to SPDE \eqref{fkpp1}, where ``semi-discrete" refers to the fact that the graph $\Gamma$ is discretized into demes but populations in the demes are infinite. 
Our scheme  utilizes  the random walk $X^n$ defined by
\eqref{RW_DF}-\eqref{RW_sym}.
\begin{enumerate}[leftmargin=*]
	\item[1.] {\bf Specify step size. } Fix a sequence $\{h_n\}_{n\geq 1}\subset (0,\infty)$   which tends to 0.
	\item[2.] {\bf Discretize $\Gamma$. } Construct $\Gamma^n=\cup_e e^n$ as in Section \ref{S:BVM},  but now with $\frac{1}{L^e_n}=h_n$ for all $e\in E$. For  $x\in e^n$ that is not adjacent to any vertex in $V$, we 
	let $I_{x}$ be the connected open interval $\{y\in \Gamma:\,d(y,x)<\frac{1}{2L^e}\}$.
	\item[3.] {\bf Independent Brownian motions. } For all  $x\in e^n$ that is not adjacent to any vertex in $V$, we let $B_x=(B_x(t))_{t\geq 0}$ be a standard  Brownian motion and that these Brownian motions $\{B_{x}\}_{x}$ are independent.
	\item[4.] {\bf Interacting SDE $\{U_x:=U^n_x\}_{x\in \Gamma^n}$. } 		Consider the system of  SDEs	
	\begin{equation}\label{U1}
	dU_x(t)  = \Big[\mathcal{L}_n U_x + \beta_e\, U_x(1-U_x)\Big] \, dt + \sqrt{\gamma_e\,L^e\,U_x(1-U_x)} \, dB_{x}(t)
	\end{equation}
	whenever $x\in e_n$ is not adjacent to any vertex in $V$, and 
	\begin{align}\label{U2}
	dU_x(t)  =& \left[\mathcal{L}_n U_x +  L^e\,\frac{\hat{\beta}(v)}{deg(v)}\, U_x(1-U_x) \right] \, dt
	\end{align}
	whenever $x\in e_n$ is adjacent to $v\in V$, where  $\mathcal{L}_n$ is the generator of the CTRW $X^n$ defined in Remark \ref{Rk:RW}.
\end{enumerate}

Our second result says that  the interacting SDEs \eqref{U1}-\eqref{U2} and the SPDE \eqref{fkpp1} are close  at least in distribution.
Write $U^n_x=\big(U^n_x(t)\big)_{t\in\geq 0}$ and let  $\{U^n_x\}_{x\in \Gamma^n}$  be a weak solution of \eqref{U1}-\eqref{U2}  for the It\^o SDE.
We further interpolate over space to define $U^n_x$ for all $x\in \Gamma$, as was done for \eqref{approx_den}. Let $\mathcal{C}([0,\infty),\,\mathcal{C}_{[0,1]}(\Gamma))$ be the space of continuous functions from $[0,\infty)$ to $\mathcal{C}_{[0,1]}(\Gamma)$, equipped with the  topology of local uniform convergence.


\begin{theorem}\label{prop:SDE}
	Let $(\Gamma,\,d)$ be a metric graph satisfying Assumption \ref{A:graph}. Let $\alpha,\beta$ and $\gamma$ be piecewise constant functions as in Example \ref{Eg:fkpp}, and  $\hat{\beta}:V\to [0,\infty)$ be a non-negative bounded function on $V$.	
	Suppose Condition (a) of Theorem \ref{T} and Assumption \ref{A:C} hold. Let $u(t,x)$ be the weak solution of SPDE \eqref{fkpp1} with initial condition $u_0\in \mathcal{C}_{[0,1]}(\Gamma)$, and $\{U^n_x\}_{x\in \Gamma^n}$ be a weak solution to  \eqref{U1}-\eqref{U2} with initial condition $\{U^n_{x}(0)=u_0(x)\}_{x\in \Gamma^n}$. Let $\{U^n_x\}_{x\in \Gamma}$  be the linear interpolation described above.
	Then as $n\to\infty$,
the  processes $U^n$ converge to $u$ in distribution in $\mathcal{C}([0,\infty),\,\mathcal{C}_{[0,1]}(\Gamma))$.
\end{theorem}

The proof of Theorem \ref{prop:SDE} will be given in Section \ref{S:SDE}.

\begin{remark}[Coupling]\rm
	The above construction gives a coupling of the interacting SDEs \eqref{U1}-\eqref{U2} and the SPDE \eqref{fkpp1} if we choose the Brownian motions to be
	\begin{equation}\label{BM_white}
	B_{x}(t):=	\sqrt{L^e}\,\dot{W}([0,t]\times I_{x}),
	\end{equation}  
	where  $\dot{W}$ is the same white noise in \eqref{fkpp1}, and if we can establish strong existence of solution for \eqref{U1}-\eqref{U2}. We do not know whether $\sup_{x\in \Gamma^n}\mathbb{E}|U^n_{x}(t)-u(t,x)|^p$ converge to zero for some $p>0$, under 
	Standard $L^p$ estimate as in Section 3 of Mueller \cite{MR1149348} does not seem to work,  due to the non-Lipschitz square root terms for $U^n$ and $u$. 
\end{remark}

\subsection{Duality for stochastic FKPP with inhomogeneous coefficients}\label{SS:DualLem}

Before turning to the proofs of Theorem \ref{T} and Theorem \ref{prop:SDE}, 
we settle the well-posedness of SPDE \eqref{fkpp1}. 
While existence of weak solution follows from tightness (Proposition \ref{T:Tight}) and \eqref{subseqLimit} for free, weak uniqueness requires a separate argument. We will establish a duality relation that implies weak uniqueness.

Duality between the standard stochastic FKPP and a branching coalescing Brownian motion is a known result due to Shiga \cite{MR948717} and \cite{MR1813840}. 
Here we generalize this result to stochastic FKPP on a metric graph $\Gamma$, with inhomogeneous coefficients and nontrivial boundary conditions. 

\begin{lemma}[Duality]\label{L:WFdual}
	Suppose that Assumptions \ref{A:graph}, \ref{A:Coe_a} and \ref{A:Coe_b} hold.
	Let $u$ be a weak solution of the SPDE \eqref{fkpp0} with initial condition $u_0\in \mathcal{C}_{[0,1]}(\Gamma)$, and $\{x_i(t):\, 1\leq i \leq n(t)\}$ be the positions at time $t$ of 
	a system of particles performing branching coalescing $\mathcal{E}$-diffusions on $\Gamma$, in which  
	\begin{itemize}
		\item branching for a particle $X_t$ occurs at rate $\beta(X_t) dt+\hat{\beta}(X_t)dL^V_t$, where $L^{V}_t$ is the local time on the set of vertices $V$ of $X$; i.e. the particle splits into two when the additive functional
		\begin{equation}\label{A^V}
			\int_0^t\beta(X_s) ds+\hat{\beta}(X_s)dL^V_s
		\end{equation}	
		exceeds an independent mean one exponential random variable. 
		\item two particles $X_t,\,Y_t$ coalesce at rate $\frac{\gamma(X_t)}{\ell(X_t)}\,dL^{(X,Y)}_t$, where $L^{(X,Y)}_t$ is the local time of $(X_t,Y_t)$ on the diagonal $\{(x,x):\,x\in \Gamma\}$; i.e. the two particles become one  when the  additive functional 
		\begin{equation}\label{A^XY}
			\int_0^t\frac{\gamma(X_s)}{\ell(X_s)}\,dL^{(X,Y)}_s
		\end{equation}
		exceeds  an independent mean one exponential random variable. 
	\end{itemize}
	Then we have the duality formula
	\begin{equation}\label{WFdual}
		\E \prod_{i=1}^{n(0)} \big(1-u_t(x_i(0))\big) =  \E \prod_{i=1}^{n(t)} \big(1-u_0(x_i(t))\big)\quad \text{for all }t\geq 0.
	\end{equation}
\end{lemma}

\medskip


The local times $L^{V}$ and $L^{(X,Y)}$ are defined as the positive continuous additive functionals (PCAF) corresponding to, respectively, the Revuz measures (Chapter 4 of \cite{MR2849840}) 
$\sum_{v\in V}\delta_v$ on $\Gamma$ and $m_*$ on $\Gamma\times \Gamma$, where $m_*$ is supported on the diagonal defined by $m(A)=m_*\{(x,x):x\in A\}$. With this new dual process, a proof of \eqref{WFdual} follows by a modification of Section 8.1 of \cite{MR3582808} and heat kernel estimates of $X$ (Theorem \ref{T:DiffusionHK}); more details of the proof are given in the Appendix.


Lemma \ref{L:WFdual} is useful to obtaining distributional properties of \eqref{fkpp0}. In particular it yields the following weak uniqueness.
\begin{lemma}[Weak uniqueness]\label{WellposeFKPP}
	Suppose  Assumptions \ref{A:graph}, \ref{A:Coe_a} and \ref{A:Coe_b} hold. Then there exists a unique weak solution to SPDE \eqref{fkpp0}. In particular, \eqref{fkpp1} has a unique weak solution.
\end{lemma}

\medskip

\begin{flushleft}
	{\bf\large Acknowledgements. }The author thanks Rick Durrett, Tom Kurtz, Carl Mueller, Edwin Perkins, Pierre Del Moral and John Yin for enlightening discussions. He thanks the referees for their helpful comments that improve the presentation and the accuracy of the results. This research is partially supported by the NSF grants DMS--1804492 and DMS--1149312 and ONR grant N00014-20-1-2411.
\end{flushleft}

\bigskip

\noindent
{\bf \Large Proofs. }The rest of the paper is devoted to the proofs of Theorems \ref{T} and  \ref{prop:SDE}. We start with some uniform heat kernel estimates of independent interest.


\clearp 
\section{Uniform heat kernel estimates}\label{S:DHK}

In this Section, we establish some uniform estimates for the transition densities for both the random walks $X^n$ (defined by
\eqref{RW_DF}-\eqref{RW_sym}) and the $\mathcal{E}$-diffusion $X$. The key point here is that the constants involved do not depend on $n$, which in particular implies the local CLT. Besides having independent interest, these results are essential to the proof of tightness in Theorem \ref{T} and of Theorem  \ref{prop:SDE}. 


\subsection{Invariance principle}

As pointed out in Remark \ref{Rk:RW}, under \eqref{Cond_C} we have
the following generalization of Donsker's invariance principle. 
Recall that $\nu$ defined in \eqref{Def:nu} is equal to $m$ when $\ell=1$.
\begin{lemma}(Invariance principle)\label{Invar}
	Let $(\Gamma,\,d)$ be a metric graph satisfying Assumption \ref{A:graph} and $\ell, \alpha$ be given in Example \ref{Eg:fkpp}.
	Suppose  Condition (a) of Theorem \ref{T} and Assumption \ref{A:C} hold. 	
	Suppose, as $n\to\infty$, the law of $X^n_0$ converges weakly to $\mu$. Then for all $T\in (0,\infty)$
	the random walks $X^n$ converge in distribution in the Skorohod space $D([0,T],\Gamma)$ to $X$, the $m$-symmetric $\mathcal{E}$-diffusion  with initial distribution $\mu$.
\end{lemma}
Lemma \ref{Invar} follows from the more general result \cite[Theorem 1]{MR3630284}.  
The latter gives invariance principles for random walks on random trees based on Dirichlet form methods. 

\subsection{Discrete heat kernel}

Let $p^n(t,x,y)$ be the transition density  of the random walk $X^n$ defined in\eqref{RW_DF}-\eqref{RW_sym},  with respect to its measure $m_n$. That is,
\begin{equation}\label{Def:DiscreteHK}
	p^n(t,x,y):=\frac{\P(X^n_t=y\,|X^n_0=x)}{m_n(y)}.
\end{equation}
Then  $p^n(t,x,y)=p^n(t,y,x)$ for all $t\geq 0$, $x,y\in \Gamma^n$. 
Recall from Remark \ref{eps_n} that  $\eps_n$ is a representative rate at which every $1/L^e$ tends to zero.

\begin{theorem}\label{T:RWHK}
	Suppose  Assumptions \ref{A:graph}, \ref{A:Coe_a} and \ref{A:C} and Condition (a) of Theorem \ref{T} hold. Then the transition densities $p^n(t,x,y)$ enjoy the following uniform estimates: For any $T\in (0,\infty)$, there exist positive constants $\{C_k\}_{k=1}^7$ and $\sigma$ such that  for all $n\in \mathbb{N}$ and $x,y \in \Gamma^n$, the followings hold.
	\begin{enumerate}
		\item (``Gaussian" upper bounds)
		\begin{align}
			p^n(t,x,y) &\leq \dfrac{C_1}{\eps_n\vee t^{1/2}}\,\exp\left(-\, C_2 \frac{|x-y|^2}{t}\right)
			\quad \hbox{for }  t\in [\eps_n,T]\quad\text{and} \label{e:2.14}\\
			p^n(t,x,y) &\leq \dfrac{C_3}{\eps_n\vee t^{1/2}}\,\exp\left(-\,C_4 \frac{|x-y|}{ t^{1/2}}\right)
			\quad \hbox{for } t\in (0, \epsilon_n]; \label{e:2.15}
		\end{align}
		\item (Gaussian lower bound)
		\begin{equation}\label{e:2.19}
			p^n(t,x,y) \geq \dfrac{C_5}{\eps_n\vee t^{1/2}}\,\exp\left(- C_6 \frac{|x-y|^2}{t}\right) \quad \hbox{for } t\in (0, T];
		\end{equation}
		\item (H\"older continuity)
		\begin{equation}\label{E:HolderCts2}
			|p^n(t,x,y)-p^n(t',x',y')| \leq C_7\, \dfrac{ (\,|t-t'|^{1/2}+ |x-x'|+ |y-y'|\,)^\sigma }
			{(t\wedge t')^{(1+\sigma)/2}}
		\end{equation}
		for all $(t,x,y),\,(t',x',y')\in (0,T]\times \Gamma^n\times \Gamma^n$.
	\end{enumerate}
\end{theorem}




\begin{remark}\rm
	It is known (for instance \cite{MR1681641, MR3678472}) that the standard Gaussian upper bound fails to hold  for small time for continuous time random walks, but we can use the weaker estimate \eqref{e:2.15}. This small time caveat is not present for the diffusion. See Theorem \ref{T:DiffusionHK}.
\end{remark}

\begin{proof}
	The proof is an application of the famous De Giorgi-Moser-Nash theory to the metric graph setting.
	It is known, for  manifold \cite{MR1354894} and for discrete graphs \cite{MR1681641}, that the two sided Gaussian estimates for reversible Markov chains is equivalent to the parabolic Harnack inequality (PHI), that these estimates are characterized by  geometric conditions, namely volume doubling (VD) plus Poincar\'e inequality (PI), and that (PHI) implies  Holder inequality. See \cite{MR2074701} for a review of results in this area.


By \cite[Theorem 1.7]{MR1681641}, all these estimates hold for $p_n$ for fixed $n$, but the key point here is that the {\it constants involved  do not depend on $n$}. 
Recall that Assumptions \ref{A:graph}, \ref{A:Coe_a} and \ref{A:C} and Condition (a) of Theorem \ref{T} are enforced throughout. These clearly give
\begin{equation}\label{GeomCondition_n}
\inf_{n\geq 1}\inf_{(x,y):x\sim y}\frac{C^n_{x,y}}{C^n_x}>0 \quad\text{where }C^n_x:=\sum_{z:z\sim x}C^n_{x,z},
\end{equation}	
which implies that the geometric condition $\Delta(\alpha)>0$ in \cite[Theorem 1.7]{MR1681641} holds uniformly in $n$.

For the rest of the proof we argue that a uniform (in $n$) PHI holds and that PHI implies all the desired estimates for $p^n$.
For simplicity, we assume $L^e_n=L_n$ for all $e\in E$. Then $1/L_n$ is the common rate at which $1/L^e$ tends to zero. We also assume the functions $\alpha$ and $\ell$ are positive constants.
The proof for Theorem \ref{T:RWHK} remains valid without this simplification: the constants will be different but still independent of $n$.

Now  $m_n(x)=1/L_n$ for all $x\in \Gamma^n$ and $n\geq 1$. Furthermore,  $C^n_{x,y} \asymp L_n$
in the sense that for some constants $\kappa,\,\tilde\kappa\in (0,\infty)$ independent of $x,y\in\Gamma^n$ and $n\geq 1$, we have
$\kappa L_n \leq C^n_{xy} \leq \tilde\kappa L_n$.  From this, we have
\begin{equation*}\label{Volume_n}
Vol_n(A) \asymp L_n^2\,m_n(A),
\end{equation*}	
where $Vol_n(A)=\sum_{x\in A}C^n_x$ is the ``volume" in the graph setting in \cite{MR1681641}. 
Denote by $B(x,r):=\{y\in \Gamma:\,d(x,y)<r\}$ a ball of the original metric graph $\Gamma$.
	Note that our ball $B(x,r)\cap \Gamma^n$ is approximately a rescaled version of the usual and more discrete  notion as that in \cite{MR1681641}. Precisely, for $x\in \Gamma^n$, we let	
	$$B_n(x,r):= \{y\in \Gamma^n:\; \#(x,y)<r \}$$
	and $\#(x,y)$ is the number of edges in $\Gamma^n$ in a shortest path connecting $x$ to $y$. Then due to \eqref{dxv}, for $r>3/L_n$ and $n$ large enough such that $L_n>2$, we have
	\begin{equation}\label{Balls}
	B_n\big(x,r (L_n -2)\big)\subset B(x,r)\cap \Gamma^n  \subset  B_n\big(x,r (L_n+2)\big).
	\end{equation}	
	
	Using \eqref{Volume_n}, \eqref{Balls} and our standing assumptions,
	we can verify the followings two geometric conditions while keeping track of the constants:  volume-doubling property  and the Poincar\'e inequality. That is,  there are constants $\kappa_1,\,\kappa_2>0$ such that {\it for all }$n\geq 1$, we have respectively
	\begin{equation}\label{VD_RW}\tag{VD[$\kappa_1$]}
		m_n(B(x,2r)\cap\Gamma^n)\leq \kappa_1\,m_n(B(x,r)\cap \Gamma^n) 
	\end{equation}
	for all $x\in \Gamma^n$ and $r>0$, and 
	\begin{equation}\label{PI_RW}\tag{PI[$\kappa_2$]}
		\sum_{y\in B(x,r)\cap \Gamma^n}|f(y)-\bar{f}_{B}|^2\,m_n(y) \leq\, \kappa_2 \,r^2\,\mathcal{E}_n(f\,1_{B(x,2r)}) 
	\end{equation}
	for all functions $f:\,\Gamma^n\to \RR$, points $x\in \Gamma^n$ and $r>0$, where $\bar{f}_{B}:=m_n^{-1}(B)\sum_{z \in B}f(z)\,m_n(z)$ is the average value of $f$ over $B=B(x,r)\cap \Gamma^n$.  
		
Equipped with \eqref{VD_RW} and \eqref{PI_RW} as well as \eqref{GeomCondition_n}, the Moser iteration argument as detailed in  \cite[Section 2]{MR1681641} yields the parabolic Harnack inequality  for some constant $\kappa_{\mathcal{H}}>0$ that is independent of $n$. That is, for all $x_0\in \Gamma$, $s\in \mathbb{R}$, $r>0$ and non-negative solution $U_n$ of the parabolic equation $\partial_t U_n=\mathcal{L}_nU_n$ on the space-time rectangle $Q:=[s,\,s+r^2]\times B_n(x_0,rL_n)$, we have
	\begin{equation}\label{PHI_pn}\tag{PHI[$\kappa_{\mathcal{H}}$]}
		\sup_{Q-}U_n  \;\leq \; \kappa_{\mathcal{H}}\, \inf_{Q+} U_n,
	\end{equation}
where $Q_{-}:=[s+ \frac{r^2}{4},\,s+ \frac{r^2}{2}]\times B_n(x_0,\,\frac{r}{2}L_n)$ and $Q_{+}:=[s+ \frac{3r^2}{4},\,s+ r^2]\times B_n(x_0,\,\frac{r}{2}L_n)$. 

The proof of \eqref{PHI_pn} follows from \cite[Section 2]{MR1681641} which is an adaptation of \cite{MR1354894}. Roughly,
\eqref{VD_RW}, \eqref{PI_RW} and \eqref{GeomCondition_n} imply a weighted Poincar\'e inequality \cite[Proposition 2.2]{MR1681641} and a Sobolev inequality
\cite[Proposition 2.4]{MR1681641}. The Sobolev inequality and \eqref{VD_RW} yield mean value inequalities for positive sub- and super-solutions \cite[Lemma 2.7]{MR1681641}. The weighted Poincar\'e inequality, on other hand, yields information on the size of $\log v$  in \cite[Lemma 2.8]{MR1681641} for any positive super-solution $v>0$. This information and \cite[Lemma 2.9]{MR1681641} allows us to stick together mean value inequalities on $v$ and $1/v$, giving the parabolic Harnack inequality for positive super-solutions $v>0$. This and the mean value inequality for sub-solutions finally give the desired \eqref{PHI_pn} for non-negative solutions.  All the constants involved in these inequality are independent of $n\geq 1$. 
Since the argument is standard, we refer the reader to \cite[Section 2]{MR1681641} for the details of the proof.

\smallskip

Following  standard arguments such as those in \cite[Section 3.3]{MR1354894} or \cite{stroock1988diffusion}, one can  show that \eqref{PHI_pn} implies the desired inequalities \eqref{e:2.14}-\eqref{E:HolderCts2}. Indeed, the Gaussian upper bounds \eqref{e:2.14}-\eqref{e:2.15} follow from  \cite[Propositions 3.1 and 3.4]{MR1681641}, since the constants in the latter results depend only on $\kappa_{\mathcal{H}}$. 
\cite[Propositions 3.1]{MR1681641} also gives the on-diagonal lower bound that is uniform in $n$. This uniform (in $n$) on-diagonal bound then gives the desired Gaussian lower bound \eqref{e:2.19} by a standard chaining argument \cite[page 329]{stroock1988diffusion}.
Finally, the H\"older inequality \eqref{E:HolderCts2} is implied by the Gaussian bounds \eqref{e:2.14}-\eqref{e:2.19} by a standard oscillation argument \cite[Theorem II.1.8-Corollary II.1.9]{stroock1988diffusion}.
\end{proof}


The uniform H\"older continuity \eqref{E:HolderCts2} together with the invariance principle imply the following local central limit theorem. 

\begin{theorem}\label{T:LCLT_CTRW}(Local CLT)
	Suppose all assumptions in Theorem \ref{T:RWHK} hold. Then 
	$$\lim_{n \to\infty} \sup_{t\in I} \sup_{x,y \in K\cap \Gamma^n}\Big|p^{n}(t,x,y)\,-\,p(t,x,y)\Big| =0 $$
	for any compact interval $I\subset (0,\infty)$ and compact subset $K\subset \Gamma$, where $p$ is the transition density \eqref{Def:p} of an $\mathcal{E}$-diffusion with $\ell$ and $\alpha$ given in Example \ref{Eg:fkpp}.
\end{theorem}

\begin{proof}
The proof follows from a compactness argument, as detailed in \cite[Theorem 2.12]{MR3678472}. 
For each $n\geq 1$ and $t>0$, we extend $p^{n}(t,\cdot,\cdot)$ to a continuous non-negative function on $\Gamma\times \Gamma$ by interpolation; see how it was done for \eqref{approx_den}. We then consider the family $\{t^{1/2}p^n\}_{n\geq 1}$ of continuous functions on $(0, \infty)\times \Gamma \times \Gamma$. The Gaussian upper bound and the H\"older continuity in Theorem \ref{T:RWHK} give us uniform pointwise bound and equi-continuity respectively. Hence, by the Arzela-Ascoli Theorem,   any sub-sequence has a further sub-sequence that converges to a continuous function $q:\,(0,\infty)\times \Gamma\times \Gamma\longrightarrow [0,\infty)$ 
locally uniformly. Lemma \ref{Invar} and the continuity of the transition density $p$ then characterize the limit.
\end{proof}

Related results and ideas can be found in Croydon and Hambly \cite{MR2453564} who investigated general conditions under which the local CLT for random walks on graphs is implied by weak convergence. 

\subsection{Heat kernel for diffusions on $\Gamma$}\label{SS:HK}


H\"older continuity of $p(t,x,y)$ then follows directly from the local CLT and \eqref{E:HolderCts2}. The two-sided Gaussian bounds for $p(t,x,y)$ do {\it not} directly follow from Theorem \ref{T:RWHK}, but 
we can establish them in the same way. 
In fact we will establish these estimates for general $\mathcal{E}$-diffusion rather than only those in Example \ref{Eg:fkpp}: 
The volume-doubling property we need for the diffusion is exactly stated in Assumption \ref{A:graph}.  Recall that $\nu(dx)=\ell(x) m(dx)$. The PI in Assumption \ref{A:graph}  implies that
\begin{align}\label{PI_dif2}
	\int_{B(x,r)}|f(y)-\bar{f}_{B}|^2\,\nu(dy) &\leq (\sup_{\Gamma}\ell) \,\int_{B(x,r)}|f(y)-\bar{f}_{B}|^2\,m(dy) \notag\\ 
	&\leq \frac{C_{PI}\,(\sup_{\Gamma}\ell) \,}{\inf_{\Gamma} \alpha} \,r^2\,\mathcal{E}\big(f\,1_{B(x,2r)}\big)
\end{align}
for all $f\in W^{1,2}(\Gamma,m)$, $x\in \Gamma$, $r>0$. With the VD \eqref{VD_dif} and PI \eqref{PI_dif2}, it is well-known (see for instance \cite{MR2962091, gyrya2011neumann} and the references therein) that the transition density $p(t,x,y)$ satisfies the parabolic Harnack inequality, which is equivalent to two sided Gaussian estimates and implies the H\"older continuity.

We summarize these important properties about transition densities $p(t,x,y)$ of diffusions on graphs in the following theorem. These properties applies to the diffusions considered in \cite{MR1245308,MR1743769, MR3634278,cerrai2019fast}.

\begin{theorem}\label{T:DiffusionHK}
	Suppose  Assumptions \ref{A:graph} and \ref{A:Coe_a} hold.
	Then the transition density of the $\mathcal{E}$-diffusion on $\Gamma$, defined in \eqref{Def:p}, enjoys the following properties: For any $T\in (0,\infty)$, there exist positive constants $\{C_k\}_{k=1}^5$ and $\sigma$  such that we have
	\begin{enumerate}
		\item (Two-sided Gaussian bounds)
		\begin{equation}\label{DiffusionGE}
			\dfrac{C_1}{t^{1/2}}\,\exp\left(- C_2 \frac{|x-y|^2}{t}\right) \leq p(t,x,y) \leq \dfrac{C_3}{t^{1/2}}\,\exp\left(- C_4 \frac{|x-y|^2}{t}\right) 
		\end{equation}
		for all  $x,y \in \Gamma$ and $t\in (0,T]$; and
		\item (H\"older continuity)
		\begin{equation}\label{DiffusionHolder}
			|p(t,x,y)-p(t',x',y')| \leq C_5\, \dfrac{ (\,|t-t'|^{1/2}+ |x-x'|+ |y-y'|\,)^\sigma }
			{(t\wedge t')^{(1+\sigma)/2}}
		\end{equation}
		for all $(t,x,y),\,(t',x',y')\in (0,T]\times \Gamma\times \Gamma$.
	\end{enumerate}
\end{theorem}


Theorem \ref{T:DiffusionHK} implies many useful properties of the diffusion $X$, including exit time estimates and strong continuity of the semigroup $\{P_t\}$ on $C_b(\Gamma)$, the space of bounded continuous functions with local uniform norm. 

\section{Proof of Theorem \ref{T}}\label{S:Pf1}

Equipped with Theorems \ref{T:RWHK}-\ref{T:DiffusionHK}, we can follow the outline of the proofs in \cite{MR1346264, MR3582808} to finish the proof of Theorem \ref{T}. 
We shall emphasize new terms and new difficulties that did not appear in \cite{MR1346264, MR3582808} in our calculations.
The dynamics of $(\xi_t)_{t\geq 0}$ is concisely described by the equation
\begin{align}\label{xieq}
	\xi_t(z) = \xi_0(z)  &+ \sum_{w\sim z} \int_0^t (\xi_{s-}(w) - \xi_{s-}(z)) \, dP^{z,w}_s \notag\\
	& + \sum_{w\sim z} \int_0^t \xi_{s-}(w)(1 - \xi_{s-}(z)) \, d\tilde P^{z,w}_s.
\end{align}
In the  space-time graphical representation (see, for instance \cite[Fig. 1]{MR3582808}), we draw an arrow $z \leftarrow w$ when there is a jump for the Poisson processes.

\subsection{Approximate martingale problem}\label{S:ApproxMtg}

For functions $f,g:\,\Gamma^n\to \RR$, we write 
\begin{equation}\label{Def:innerprod}
	\<f,g\>_{e}:= \frac{1}{L^e}\sum_{x\in e^{n}} f(x)g(x) \quad\text{and}\quad
	\<f,g\>:=\<f,g\>_{m_n}:=\sum_{e}\<f,g\>_{e}.
\end{equation} 
We also identify $f$ with a function on $\Lambda_n$ by setting $f(z):=f(x)$ when $z=(x,i)$.

Let $\phi: [0,\infty)\times \Gamma \to \RR$ be a continuous function that is continuously differentiable in $t$, twice continuously differentiable and has compact support in $x\in \Gamma$ and satisfies the boundary condition 
\begin{equation}\label{glue_phi}
	\big(\nabla_{out}\,\phi_t\,\cdot [\alpha]\big)(v)=0 \quad\text{for }t\geq 0,
\end{equation}
where we write $\phi_t(x)=\phi(t,x)$ interchangeably.


Applying integration by parts to $\xi_t(z) \phi_t(z)$, using \eqref{xieq}, and summing over $x$, we obtain for all $t>0$ and edge $e$, we have
\begin{align}
	\<u^n_t,\,\phi_t\>_e- & \<u^n_0,\,\phi_0\>_e  -\int_0^t \<u^n_s,\partial_s\phi_s\>_e\,ds  
	\nonumber\\
	&= (M^eL^e)^{-1} \sum_{z\in \Lambda^{e}_n} \sum_{w\sim z} \int_0^t (\xi_{s-}(w) - \xi_{s-}(z)) \phi_s(z) \, dP^{z,w}_s
	\label{term1}\\
	& + (M^eL^e)^{-1} \sum_{z\in \Lambda^{e}_n} \sum_{w\sim z} \int_0^t \xi_{s-}(w)(1 - \xi_{s-}(z)) \phi_s(z) \, d\tilde P^{z,w}_s.
	\label{term2}
\end{align}
Note that $w$ in \eqref{term1} and \eqref{term2} can be on a different edge $\tilde{e}\in E(v)\setminus e$ if $z$ is adjacent to a vertex $v$. 

\medskip

\noindent
{\bf Outline. }
To describe the local behavior near a vertex and to simplify the presentation, we first consider the case when $\Gamma$ consists of $deg(v)$ positive half real lines starting from a single common vertex $v$. For each $e\in E(v)$ we enumerate the set $e^n$ as $(x_1^e,\,x_2^e,\,x_3^e,\cdots)$ along the direction of $e$ away from $v$, so $x_1^e$ is closest to $v$. We also identify $x_{k}^e+(L^e)^{-1}$ with $x_{k+1}^e$. See {\bf Figure \ref{Fig:LatticeGn}}.
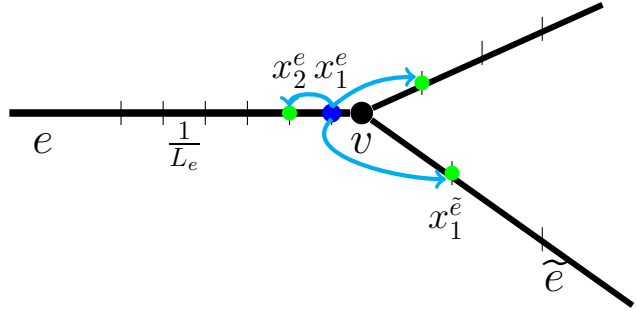
\begin{figure}
	\begin{center}
		\begin{tikzpicture}[scale=0.8]
		
		\node (O) at (0,4) {};
		
		\node[big dot] (V1) at (6,4) {};
		\node at (6,3.5) {\LARGE $v$};
		
		
		\draw [line width=1mm] (O) -- (V1);
		\draw [line width=0.8mm] (V1) -- (10,5.8);
		\draw [line width=0.8mm] (V1) -- (10.5,0.8);
		
		\node at (3.05,3.45) {\Large $\frac{1}{L_e}$};
		\node at (0.7, 3.5) {\LARGE $e$};
		\node at (9.2, 1.3) {\LARGE $\widetilde{e}$};
		\node at (10, 0.5) {};
		
		\draw [-] (2,3.8) -- (2,4.2);
		\draw [-] (2.7,3.8) -- (2.7,4.2);
		\draw [-] (3.4,3.8) -- (3.4,4.2);
		\draw [-] (4.1,3.8) -- (4.1,4.2);
		\draw [-] (4.8,3.8) -- (4.8,4.2);
		\draw [-] (5.5,3.8) -- (5.5,4.2);
		
		\draw node[fill=blue, circle, inner sep=0pt, minimum size=2.5mm] at (5.5,4) {};
		\draw node[fill=green, circle, inner sep=0pt, minimum size=2mm] at (4.8,4) {};
		
		\draw [-] (7,4.3) -- (7,4.7);
		\draw [-] (8,4.8) -- (8,5.2);
		\draw [-] (9,5.2) -- (9,5.6);
		
		\draw node[fill=green, circle, inner sep=0pt, minimum size=2mm] at (7,4.5) {};
		
		
		\draw [-] (7.5,2.8) -- (7.5,3.2);
		\draw [-] (9,1.7) -- (9,2.1);
		
		\draw node[fill=green, circle, inner sep=0pt, minimum size=2mm] at (7.5,3) {};
		
		\draw [->,ultra thick, cyan] (5.5,4.1) to [out=90,in=90] (4.8,4.1);
		\draw [->,ultra thick, cyan] (5.5,4.1) to [out=40,in=180] (6.9, 4.6);
		\draw [->,ultra thick, cyan] (5.5,3.9) to [out=225,in=180] (7.4,2.9);
		
		\node at (5.5,4.7) {\Large $x^e_1$};
		\node at (4.8,4.7) {\Large $x^e_2$};
		\node at (7.4,2.25) {\Large $x^{\tilde{e}}_1$};
		
		\end{tikzpicture}		
		\caption{Illustration of demes $x_k^e$ on the discretized graph $\Gamma^n$ \label{Fig:LatticeGn}}
	\end{center}
\end{figure}	
Observe that when both $z$ and $w$ are on $e$ but {\it not} in $x^e_1$, then $a^{z,w}=a^{w,z}$. Hence we split the double sum in \eqref{term1} as
\begin{equation}\label{splitterm}
	\sum_{z\in\Lambda^e_n} \sum_{w\sim z} {\bf 1}_{\{w,z\notin x_1^e\}} + \Big(\sum_{z\in x_1^e} \sum_{w\sim z} + \sum_{z\in x_2^e} \sum_{w \in x_1^e}\Big) \;:=\;\eqref{term1}(i)\,+\,\eqref{term1}(ii),
\end{equation}
where $\eqref{term1}(ii)$ is the term $(\sum_{z\in x_1} \sum_{w\sim z} + \sum_{z\in x_2} \sum_{w \in x_1})$. 
The sum in $\eqref{term1}(i)$  is symmetric and 
we shall makes use the gluing condition \eqref{glue_phi} to treat this term; while treatment of the (boundary) term $\eqref{term1}(ii)$ required the choices of parameters in Conditions (a)-(c) of Theorem \ref{T}. The next 3 subsections are computations for terms $\eqref{term1}(i)$, $\eqref{term1}(ii)$ and $\eqref{term2}$ respectively, which lead to the  martingale problem for $u^n$.

As we shall see, since $\phi$ has compact support and the branch lengths of $\Gamma$ is bounded below by a positive constant (Assumption \ref{A:graph}), all sums involving $\phi$ are finite sums and the proofs in Subsections \ref{S:ApproxMtg}-\ref{S:tight} work equally well for any graph satisfying Assumption \ref{A:graph}. 

\subsubsection{Term \eqref{term1}(i): White noise, Laplacian and Glueing condition}

For $\eqref{term1}(i)$, some ideas in \cite[Section 5]{MR3582808} can be reused which we now briefly describe. 
Let $\xi^c_t(z) := 1- \xi_t(z)$ and rewrite  the integrand of \eqref{term1} as
\begin{align}
	&(\xi_{s-}(w) - \xi_{s-}(z))\, \phi_s(z) \notag\\
	=\;&  [\xi_{s-}(w)\,\xi^c_{s-}(z) - \xi_{s-}(z)\,\xi^c_{s-}(w)]\,\phi_s(z) \notag \\ 
	=\;& \xi_{s-}(w)\,\xi^c_{s-}(z) \phi_s(w) - \xi_{s-}(z)\,\xi^c_{s-}(w) \phi_s(z) \label{ig1} \\
	& + \xi_{s-}(w)\,\xi^c_{s-}(z)\, \big(\phi_s(z) - \phi_s(w)\big). \label{ig2}
\end{align}

As we will see, the white noise will come from \eqref{ig1} while the Laplacian term and the gluing term come from \eqref{ig2}. To simplify notation we denote $A^n_{\tilde{e},\,e}:=2\,C^n_{\tilde{e},e}L^e/M^e$, so $a^{z,w}=A^n_{\tilde{e},\,e}$ in Condition (c) in the statement of Theorem \ref{T}.

\medskip

{\bf White noises. } 
We first work with \eqref{ig1}. 		
Interchanging the roles of $z$ and $w$ in the first double sum
and writing $Q^{z,w}_s = P^{w,z}_s - P^{z,w}_s$, this part of $\eqref{term1}(i)$ becomes a martingale 
\begin{equation}\label{term1i}
	Z^e_t(\phi):=(M^eL^e)^{-1} \sum_{z\in\Lambda^e_n} \sum_{w\sim z} {\bf 1}_{\{w,z\notin x_1\}}\int_0^t  \xi_{s-}(z) \xi^c_{s-}(w) \phi_s(z) \, dQ^{z,w}_s.
\end{equation}
Since $a^{z,w}=a^{w,z}$ and  the quadratic variance process $\langle Q^{z,w} \rangle_t = 2A^n_{e,e} t$, we have quadratic variation $\<Z^e(\phi)\>_t$
converges to
$$ 
\gamma_e \int_0^t  \int_{e}  u_s(x)\,(1-u_s(x)) \phi_s(x)^2 \,m(dx) \, ds
$$
since $A^n_{e,e}/L^e\to\gamma_e/4$ by Conditions (a)-(c) and by the smoothness assumption of $\phi$, assuming $\mathcal{C}$-tightness of $\{u^n\}$.  Details for this part is the same as those in \cite[Section 5]{MR3582808}.
{\it Here, and in what follows in this Subsection, the claimed convergences follow once we have proved $\mathcal{C}$-tightness. See Subsection \ref{S:tight} for the proof, which does not use any of the convergences claimed in this Subsection. See Remark \ref{Rk:testfcn}.}

\medskip

{\bf Laplacian and gluing condition. } Next we work with \eqref{ig2}. We write $L=L^e$ and $M=M^e$ for simplicity when there is no confusion, and denote the discrete gradient and the discrete Laplacian respectively by
\begin{align}
	\nabla_Lf(x)& :=L\,\Big(f(x+L^{-1})-f(x)\Big) \qquad \qquad\qquad\quad \text{if } x\in e^n,\label{dnabla}\\
	\Delta_Lf(x)&:= L^2\,\Big(f(x+L^{-1})+f(x-L^{-1})-2f(x)\Big) \quad \text{if } x\in e^n\setminus\{x^e_1\}. \label{dLaplacian}
\end{align}	


We break \eqref{ig2} into the average terms and a fluctuation term
\begin{align}
	& (ML)^{-1} \sum_{z \in\Lambda^e_n} \sum_{w\sim z}{\bf 1}_{\{w,z\notin x_1\}} \int_0^t \xi_{s-}(w) \xi^c_{s-}(z) [\phi_s(z) - \phi_s(w)] \, A^n_{e,e} \, ds
	\label{term3a}\\
	& + (ML)^{-1} \sum_{z \in\Lambda^e_n} \sum_{w\sim z}{\bf 1}_{\{w,z\notin x_1\}} \int_0^t \xi_{s-}(w) \xi^c_{s-}(z) [\phi_s(w) - \phi_s(z)]  ( dP^{z,w}_s - A^n_{e,e} \, ds).
	\label{term3b}
\end{align}
We can replace $\xi^c_{s-}$ by 1 in \eqref{term3a} without changing its value, 
because by symmetry (valid since we are summing over the {\it same} set for $z$ and $w$), 
$$\sum_z \sum_{w\sim z}\xi_{s-}(w) \xi_{s-}(z) [\phi_s(z) - \phi_s(w)] =0\qquad\text{for all } s>0.$$ 

Doing the double sum over $w$ and then over $z\sim w$, and recalling that we identify $x_{k}^e+(L^e)^{-1}$ with $x_{k+1}^e$, we see that the integrand of \eqref{term3a} is
\begin{align}
	&(M^eL^e)^{-1} \sum_{z \in\Lambda^e_n} \sum_{w\sim z}{\bf 1}_{\{w,z\notin x_1^e\}}  \xi_{s-}(w)  [\phi_s(z) - \phi_s(w)] \, A^n_{e,e} \notag\\
	=&  \frac{A^n_{e,e}M}{L^3}\sum_{x \geq x^e_3} u^n_{s-}(x) \Delta_L \phi_s(x)+ \frac{A^n_{e,e}M}{L}  u^n_{s-}(x^e_2)\,[\phi_s(x^e_3)-\phi(x^e_2)] \notag\\
	=&  \frac{A^n_{e,e}M}{L^3}\sum_{x \geq x^e_2} u^n_{s-}(x) \Delta_L \phi_s(x)- \frac{A^n_{e,e}M}{L}  u^n_{s-}(x^e_2)\,[\phi_s(x^e_1)-\phi(x^e_2)] \notag\\
	=&  \frac{A^n_{e,e}M}{L^2}\,\left\{ \Big(\frac{1}{L}\sum_{x \geq x^e_2} u^n_{s-}(x) \Delta_L \phi_s(x)  \Big)+   u^n_{s-}(x^e_2)\,\nabla_L\phi_s(x^e_1) \right\}.\label{GlueT0}
\end{align}
By our assumption on $\phi$, $\Delta_L \phi_s$ converges to $\Delta\phi_s$ uniformly on compact subsets of $e$ and  $\lim_{n\to\infty}\frac{A^n_{e,e}M}{L^2}=\alpha_e$, so \eqref{term3a} converges to
\begin{equation}\label{LapBound}
	\alpha_e\int_0^t \big(\int_eu_s(x)\Delta\phi_s(x)\,dx\big) + u_{s}(v)\,\nabla_{e+}\phi_s(v)\,ds.
\end{equation}
Using the \textit{gluing assumption} \eqref{glue_phi}, we see that the last term will vanish upon summation over $e\in E(v)$.
The other term, 
\eqref{term3b}, is a martingale $E^{(1)}_t(\phi)$ with
\begin{align*}
	\langle E^{(1)}(\phi) \rangle_t	\leq &\, \frac{A^n_{e,e}}{(ML)^2} \sum_z \sum_{w\sim z} \int_0^t  (\phi_s(z) - \phi_s(w))^2  \, ds\\
	= &\, \frac{2A^n_{e,e}}{L^2} \int_0^t \<1,\, |\nabla_L\phi_s|^2\>_e  \, ds \to 0
\end{align*}
since $A^n_{e,e}/L^e \to \gamma/4$ and $L^e \to \infty$, by \eqref{Cond_C}.

\subsubsection{Term \eqref{term1}(ii): Matching boundary condition}



We now deal with $\eqref{term1}(ii)$ which are the remaining terms
$\sum_{z\in x^e_1} \sum_{w\sim z} + \sum_{z\in x^e_2} \sum_{w \in x^e_1}$ in \eqref{splitterm}. Our goal is to show that this term converges to zero under the choice for $C^n_{xy}$'s specified in condition (c) of Theorem \ref{T}.
For this we further take a summation over all edges in $E(v)$ to obtain
\begin{align}
	&\sum_{e\in E(v)}(M^eL^e)^{-1} \sum_{z\in x^e_1} \sum_{w\sim z} \int_0^t (\xi_{s-}(w) - \xi_{s-}(z)) \phi_s(z) \, dP^{z,w}_s\notag \\
	&+\sum_{e\in E(v)}(M^eL^e)^{-1}  \sum_{z\in x^e_2} \sum_{w \in x^e_1} \int_0^t (\xi_{s-}(w) - \xi_{s-}(z)) \phi_s(z) \, dP^{z,w}_s\notag\\
	=& \sum_{e\in E(v)}(M^eL^e)^{-1}\sum_{\tilde{e}\in E(v)\setminus e} \sum_{z\in x^{e}_1} \sum_{w\in x^{\tilde{e}}_1} \int_0^t (\xi_{s-}(w) - \xi_{s-}(z)) \phi_s(z) \, dP^{z,w}_s \label{boundary1}\\
	&+ \sum_{e\in E(v)}(M^eL^e)^{-1} \Big( \sum_{z\in x^e_2} \sum_{w \in x^e_1}+ \sum_{w\in x^e_2} \sum_{z \in x^e_1}\Big) \int_0^t \big(\xi_{s-}(w) - \xi_{s-}(z)\big) \phi_s(z) \, dP^{z,w}_s. \label{boundary2}
\end{align} 
We break \eqref{boundary1} into an average term and a fluctuation term
\begin{align}
	&\sum_{e\in E(v)}(M^eL^e)^{-1} \sum_{\tilde{e}\in E(v)\setminus e} \sum_{z\in x^{e}_1} \sum_{w\in x^{\tilde{e}}_1} \int_0^t (\xi_{s-}(w) - \xi_{s-}(z)) \phi_s(z) \, A^n_{e,\tilde{e}}\,ds \label{aver1}\\
	+&\sum_{e\in E(v)}(M^eL^e)^{-1}  \sum_{\tilde{e}\in E(v)\setminus e} \sum_{z\in x^{e}_1} \sum_{w\in x^{\tilde{e}}_1} \int_0^t (\xi_{s-}(w) - \xi_{s-}(z)) \phi_s(z) \,\big(dP^{z,w}_s-A^n_{e,\tilde{e}} \,ds\big).\label{fluc1}
\end{align} 

Grouping terms in {\it unordered pairs} of distinct elements in $E(v)$, the average term \eqref{aver1} is 
\begin{align}
	&\sum_{e\in E(v)}\sum_{\tilde{e}\in E(v)\setminus e}\frac{A^n_{e,\tilde{e}}}{M^eL^e}  \sum_{z\in x^{e}_1} \sum_{w\in x^{\tilde{e}}_1} \int_0^t (\xi_{s-}(w) - \xi_{s-}(z)) \phi_s(z) \,ds \notag\\
	=& \sum_{e\in E(v)}\sum_{\tilde{e}\in E(v)\setminus e} 2C^n_{e,\tilde{e}}   \int_0^t \big(u^n_{s-}(x^{\tilde{e}}_1) - u^n_{s-}(x^e_1)\big) \phi_s(x^e_1) \,ds\notag\\
	=&\int_0^t  \sum_{(e,\tilde{e})\text{unordered}}\,
	\Big(	2C^n_{e,\tilde{e}}\,-\,2C^n_{\tilde{e},e} \Big)\phi_s(x^e_1)\big(u^n_{s-}(x^{\tilde{e}}_1) - u^n_{s-}(x^e_1)\big)\,ds \,+\,err_1(t) \label{aver1_result_0}\\
	=&\,0\,+\,err_1(t),  \label{aver1_result}
\end{align} 
where
\begin{align}
err_1(t)= &\int_0^t  \sum_{(e,\tilde{e})\text{unordered}}\,
2C^n_{\tilde{e},e}\,\big(\phi_s(x^e_1)-\phi_s(x^{\tilde{e}}_1)\big)\,\big(u^n_{s-}(x^{\tilde{e}}_1) - u^n_{s-}(x^e_1)\big)\,ds. \label{err1}
\end{align}
The equality \eqref{aver1_result} says that the first sum in 
\eqref{aver1_result_0} is zero. To see this, we group the terms of this sum in pairs in such a way that $\sum_{(e,\,\tilde{e})\text{unordered}}$ is over all unordered pairs of distinct elements in $E(v)$. The last equality \eqref{aver1_result} then follows by symmetry of the conductances. 

The error term $err_1(t)$ is $o(1)$, i.e.
tends to $0$ as $n\to\infty$, since $\sup_{s\in[0,T]}\sup_{n\geq 1}L^{e}|\phi_s(x^e_1)-\phi_s(v)|<C_T$ and (by Remark \ref{eps_n})
\begin{equation}
	\sup_{n\geq 1}\sum_{e\in E(v)}\sum_{\tilde{e}\in E(v)}C^n_{\tilde{e},e}\Big(\frac{1}{L^e}+\frac{1}{L^{\tilde{e}}}\Big) <\infty. \label{Cond_err1}
\end{equation}

The fluctuation term \eqref{fluc1} has quadratic variation
\begin{align*}
	&\sum_{e\in E(v)}\sum_{\tilde{e}\in E(v)\setminus e}\frac{ A^n_{e,\tilde{e}}}{(M^eL^e)^2}  \sum_{z\in x^{e}_1} \sum_{w\in x^{\tilde{e}}_1} \int_0^t \big(\xi_{s-}(w) - \xi_{s-}(z)\big)^2\, \phi^2_s(z) \,ds\\
	\leq &\sum_{e\in E(v)}\sum_{\tilde{e}\in E(v)\setminus e} \frac{2C^n_{e,\tilde{e}}}{L^e}\,\frac{M^{\tilde{e}}}{M^e}  \int_0^t \Big(\frac{u^n_{s-}(x^{\tilde{e}}_1)}{M^e} + \frac{u^n_{s-}(x^e_1)}{M^{\tilde{e}}}\Big) \,\phi^2_s(x^e_1) \,ds \\
	\leq &\sum_{e\in E(v)}\sum_{\tilde{e}\in E(v)\setminus e} \frac{2C^n_{e,\tilde{e}}}{L^e}\,\frac{M^{\tilde{e}}}{M^e} \Big(\frac{1}{M^e} + \frac{1}{M^{\tilde{e}}}\Big) \int_0^t  \phi^2_s(x^e_1) \,ds 
\end{align*} 
which tends to $0$ as $n\to\infty$ since
\begin{equation}
	\sum_{(e,\tilde{e})}\frac{C^n_{e,\tilde{e}}M^{\tilde{e}}}{L^eM^e}\Big(\frac{1}{M^e}+\frac{1}{M^{\tilde{e}}}\Big)\leq 
	C\,\sum_{(e,\tilde{e})}\frac{M^{\tilde{e}}}{M^e}\Big(\frac{1}{M^e}+\frac{1}{M^{\tilde{e}}}\Big)
	\to 0, \label{Cond_fluc_0}
\end{equation}
by \eqref{Cond_C}.

Similarly, we break \eqref{boundary2} into an average term and a fluctuation term
\begin{align}
	&\sum_{e\in E(v)}(M^eL^e)^{-1}  \sum_{z\in x^e_2} \sum_{w \in x^e_1} \int_0^t \big[(\xi_{s-}(w) - \xi_{s-}(z)) \phi_s(z) \, a^{zw} \notag\\
	& \qquad\qquad\qquad\qquad\qquad\qquad+ (\xi_{s-}(z) - \xi_{s-}(w)) \phi_s(w) \,a^{wz}\big]\,ds \label{aver2}\\
	+&\sum_{e\in E(v)}(M^eL^e)^{-1}  \sum_{z\in x^e_2} \sum_{w \in x^e_1} \int_0^t (\xi_{s-}(w) - \xi_{s-}(z)) \phi_s(z) \,\big(dP^{z,w}_s- a^{zw} \,ds\big) \notag\\
	& \qquad\qquad\qquad\qquad\qquad\quad+(\xi_{s-}(z) - \xi_{s-}(w)) \phi_s(w) \,\big(dP^{w,z}_s- a^{wz}\,ds\big). \label{fluc2}
\end{align} 

The average term is equal to
\begin{align}
	&\sum_{e\in E(v)} (M^eL^e)^{-1}\,\frac{2C^n_{ee}L^e}{M^e}\, \sum_{z\in x^e_2} \sum_{w \in x^e_1} \int_0^t (\xi_{s-}(w) - \xi_{s-}(z)) \phi_s(z)+(\xi_{s-}(z) - \xi_{s-}(w)) \phi_s(w)\,ds \notag\\
	=&\sum_{e\in E(v)}2C^n_{ee}   \int_0^t\big(u^n_{s-}(x^e_2) - u^n_{s-}(x^e_1)\big)\, \big(\phi_s(x^e_1)-\phi_s(x^e_2) \big)\,ds \notag\\
	=& \,err_2(t)
	\label{aver2_result}
\end{align} 
where $err_2(t)$ is the $o(1)$ error term 
\begin{align}
	err_2(t)&=\,-\sum_{e\in E(v)}\frac{2C^n_{ee}}{L^e}\int_0^t\big(u^n_{s-}(x^e_2) - u^n_{s-}(x^e_1)\big)\,\nabla_L\phi_s(x^e_1) \,ds. \label{err2}
\end{align}

The variance of the fluctuation term is 
\begin{align*}
	=&\sum_{e\in E(v)}
	\frac{A^n_{e,e}}{(M^eL^e)^{2}} \sum_{z\in x^e_2} \sum_{w \in x^e_1} \int_0^t \big(\xi_{s-}(w) - \xi_{s-}(z)\big)^2 \,[\phi^2_s(x^e_2) +
	\phi^2_s(x^e_1)]
	\,ds \\
	\leq &\sum_{e\in E(v)}
	\frac{2\,A^n_{e,e}}{(L^e)^{2}}  \int_0^t [\phi^2_s(x^e_2) +
	\phi^2_s(x^e_1)]\,\,ds\notag\\
	=& \sum_{e\in E(v)}
	\frac{4\,C^n_{ee}}{L^e\,M^e}  \int_0^t [\phi^2_s(x^e_2) +
	\phi^2_s(x^e_1)]\,\,ds\to 0\notag
\end{align*}
since $1/M^e \to 0$ and $C^n_{ee}$ is of order $L^e$, as $n\to\infty$, by Remark \ref{eps_n}.

\subsubsection{Term \eqref{term2}}

{\bf Drift term. } We break \eqref{term2} into an average term and a fluctuation term	
\begin{align}
	& (M^eL^e)^{-1} \sum_{z\in \Lambda^{e}_n} \sum_{w\sim z} \int_0^t \xi_{s-}(w)(1 - \xi_{s-}(z)) \phi_s(z) \, b^{z,w}\, ds
	\label{term2a} \\
	& + (M^eL^e)^{-1} \sum_{z\in \Lambda^{e}_n} \sum_{w\sim z} \int_0^t \xi_{s-}(w)(1 - \xi_{s-}(z)) \phi_s(z)  \, (d\widetilde P^{z,w}_s - b^{z,w}\, ds ).
	\label{term2b}
\end{align}
Recalling the definition of the density $u^n$ in \eqref{approx_den}, we check that \eqref{term2a} becomes 
\begin{align}\label{Y^e_tAlso}
	&  \frac{M^e}{L^e} \sum_{x\in e^n} \int_0^t [u^n_{s-}(x-L^{-1}) ] (1-u^n_{s-}(x)) \phi_s(x) \,B^n_e\, ds \notag\\
	+\,& \frac{M^e}{L^e} \sum_{x\in e^n\setminus \{x^e_1,x^e_2\}} \int_0^t [u^n_{s-}(x+L^{-1}) ] (1-u^n_{s-}(x)) \phi_s(x) \,B^n_e\, ds \notag\\
	+\,& \frac{M^e}{L^e}  \int_0^t u^n_{s-}(x^e_1)  (1-u^n_{s-}(x^e_2)) \phi_s(x^e_2) \,  \widehat{B}^n_{e,e}\,ds \notag\\
	+\,& \sum_{\tilde{e}\in E(v)\setminus e} \frac{M^{\tilde{e}}}{L^e}  \int_0^t u^n_{s-}(x^{\tilde{e}}_1)  (1-u^n_{s-}(x^e_1)) \phi_s(x^e_1) \,  \widehat{B}^n_{e,\tilde{e}}\,ds	 
\end{align}
where $\widehat{B}^n_{e,e}=b^{x^e_2,\,x^e_1}$ and $\widehat{B}^n_{e,\tilde{e}}=b^{x^e_1,\,x^{\widetilde{e}}_1}$.
The sum of the first two terms converges to  
$$
\frac{\beta_e}{2} \int_0^t \int_{e}2 u_s(x) (1-u_s(x)) \phi_s(x) \, dx \, ds\,
$$ 
as $n \to\infty$ by Condition (d). After a further summation over $e$, we see from \eqref{Cond_betav} that the sum of the last two terms tends to 
\begin{equation}\label{Cond_betav2}
	\frac{\hat{\beta}(v)}{4}\int_0^t 2 u_s(v) (1-u_s(v)) \phi_s(v)  \, ds.
\end{equation}


The second term \eqref{term2b} is a martingale $E^{(2)}_t(\phi)$ with
\begin{align*}
	\langle E^{(2)}(\phi) \rangle_t \le &\,\frac{1}{ (M^eL^e)^{2}} \sum_{z\in \Lambda^{e}_n} \sum_{w\sim z} \int_0^t \phi^2_s(z) \,b^{z,w}\, ds\\
	\leq &\, \frac{(M^e)^2B^n_e+ M^eM^{\tilde{e}}\widehat{B}^n_{e,\tilde{e}}(|E(v)|-1)}{ (M^eL^e)^{2}}  \int_0^t \<1,\,\phi^2_s\>_e\,ds\\
	= &\, \Big(\frac{B^n_e}{ (L^e)^{2}} + \frac{ M^{\tilde{e}}\widehat{B}^n_{e,\tilde{e}}}{L^e}\, \frac{(|E(v)|-1)}{M^eL^e} \Big)\int_0^t \<1,\,\phi^2_s\>_e\,ds \to 0.
\end{align*}
since $\frac{B^n_e}{ (L^e)^{2}}\to 0$, $\frac{ M^{\tilde{e}}\widehat{B}^n_{e,\tilde{e}}}{L^e}$ is bounded and $\frac{(|E(v)|-1)}{M^eL^e}\to 0$.



\medskip

\noindent
{\bf Limiting martingale problem. }
Combining our calculations, we can characterize any sub-sequential limit $u$ of $u^n$. Suppose $u$ is the distributional limit of a sub-sequence of $u^{n}$ in 	$D([0,\infty),\,\mathcal{C}_{[0,1]}(\Gamma))$. Then there is a further sub-sequence that converges almost surely to $u$ which, by our calculation in this section 
\ref{S:ApproxMtg}, satisfies the following: 
for any $\phi\in C^{1,2}_c([0,\infty)\times \Gamma)$ which satisfies the gluing condition \eqref{glue} for all $t$,
\begin{align}\label{LimitingMtg}
\int_{\Gamma}u_t(x)&\,\phi_t(x)\,m(dx)\, -   u_0(x)\,\phi_0(x)\,m(dx)   - \int_0^t \int_{\Gamma} u_s(x)\,\partial_s\phi_s(x)\,m(dx)\,ds  
\nonumber \\
&- \int_0^t \int_{\Gamma} \alpha(x) u_s(x) \Delta \phi_s(x) -   \beta(x) \, u_s(x)(1-u_s(x))\, \phi_s(x)\,m(dx) \, ds \nonumber\\
&- \frac{1}{2}\int_0^t  \sum_{v\in V} \hat{\beta}(v) \, u_s(v)(1-u_s(v))\, \phi_s(v) \, ds
\end{align}
is a continuous martingale with quadratic variation
\begin{equation}\label{LimitingQuad}
\int_0^t \gamma(x) \int_{\Gamma} u_s(x)(1-u_s(x)) \phi^2_s(x) \,m(dx) \,ds,
\end{equation}
which is the martingale problem formulation of \eqref{fkpp1}; see p. 536-537 in \cite{MR1346264}. From this, one can construct on a probability space 
(see, for instance, \cite[Section V20]{MR1780932})
a white noise $\dot{W}$ on the Polish space $\Gamma\times[0,\infty)$  such that \eqref{E:WeakSol} holds for all $\phi\in C_c(\Gamma)\cap C^2(\mathring{\Gamma})$. Hence $u$ solves the \eqref{fkpp1} weakly. 
We have shown that, under convergence in distribution in 	$D([0,\infty),\,\mathcal{C}_{[0,1]}(\Gamma))$,
\begin{equation}\label{subseqLimit}
	\text{ any sub-sequential limit of } \{u^n\}\text{ solves the SPDE } \eqref{fkpp1} \text{ weakly.}
\end{equation}


\begin{remark}\rm\label{Rk:testfcn}
	All calculation in this section  {\it before} taking $n\to \infty$ hold for a more general class of test functions $\phi$. Namely $\phi: [0,\infty)\times \Gamma^n \to \RR$ is merely defined on $\Gamma^n$ for the spatial variable, but it is continuously differentiable in $t$ and such that all sums that appeared in the above calculations are well-defined (e.g. when $\phi$ is bounded and has compact support in $\Gamma$). In particular, the gluing condition \eqref{glue_phi} is not needed in the pre-limit calculations and it is legitimate to apply these calculations to the test function to be defined in \eqref{E:testfcn}.
\end{remark}

In the next two subsections, we establish tightness of $\{u^n\}$. Weak uniqueness of \eqref{fkpp1} (Lemma \ref{WellposeFKPP}) together with \eqref{subseqLimit} then completes the proof of Theorem \ref{T}.

\subsection{Green's function representation}\label{S:green}

Following \cite[Section 4]{MR3582808}, our proof of tightness begins with the Green's function representation of $u^n$.  This will be obtained in \eqref{E:Green_n_sim}-\eqref{E:Green_n2_sim} in this subsection. New terms that do not appear in \cite{MR3582808} will be pointed out in our derivations.

Denote $(P^n_t)_{t\geq 0}$ to be the semigroup of $X^n$, defined by 
\begin{equation}\label{SemiRW}
P^n_tf(x):=\E[f(X^n_t)|X^n_0=x]=\sum_{y\in \Gamma^n}f(y)\,p^n(t,x,y)\,m_n(y)
\end{equation}
for bounded measurable functions $f$, where we recall from \eqref{Def:DiscreteHK} that $p^n(t,x,y):=\frac{\P(X^n_t=y\,|X^n_0=x)}{m_n(y)}$.

Observe that $\<f,g\>$ defined before can be written as $\sum_{x\in \Gamma^n}f(x)g(x)m_n(x)$ and that for any $y\in \Gamma^n$,
$$u^n_t(y)= \<u^n_t,\,\sum_e{\bf 1}_{y\in e^n}\, L^e\,{\bf 1}_y\> = \<u^n_t,\,\phi_t\>,$$
where ${\bf 1}_y$ is the indicator function and
\begin{equation}\label{E:testfcn}
	\phi_s(x):=\phi^{t,y}_s(x):= 
	\begin{cases} p^{n}(t-s,x,y) & \text{for }s\in[0,t], \\ 0 & \hbox{otherwise.} \end{cases}
\end{equation}

Applying the approximate martingale problem \eqref{term1}-\eqref{term2}  with test function $\phi_s:=\phi^{t,y}_s$ in \eqref{E:testfcn} (see Remark \ref{Rk:testfcn} for why we can do this)
and using the facts that 
\begin{itemize}
	\item $\partial_s\phi_s+\alpha_e\Delta_{L^e}\phi_s =o(1)$, where $o(1)$ is a term which tends to 0 as $n\to\infty$, uniformly for $x\in e^n\setminus \{x^e_1\}$, $e\in E$ and $s\in (0,\infty)$ (see Remark \ref{Rk:RWGen}); and 
	\item $\<u^n_0,\phi^{t,y}_0\> = P^n_{t}u^n_0(y)$ for all $y\in \Gamma^n$,
\end{itemize}
we obtain
\begin{align}
	u^n_t(y) =\,& P^n_{t}u^n_0(y)+ \sum_e\frac{1}{L^e} \int_0^t u^n_s(x^e_1)\partial_s\phi_s(x^e_1)\,ds \label{E:Green_n0}\\
	&+ \sum_e \Big(Y^e_t(\phi)+ Z^e_t(\phi)+E^{(1,e)}_t(\phi) + E^{(2,e)}_t(\phi)\Big) \label{E:Green_n}\\
	&  + \sum_e \Big( T^e_t(\phi)+U^e_t(\phi)+ V^e_t(\phi)+ E^{(3,e)}_t(\phi)+ E^{(4,e)}_t(\phi)\Big) \label{E:Green_n2}
\end{align}
for $t\geq 0$ and $y\in \Gamma^n$. Here
the terms $Z^e_t(\phi)$, $E^{(1,e)}_t(\phi)$, $E^{(2,e)}_t(\phi)$,  $E^{(3,e)}_t(\phi)$ and $E^{(4,e)}_t(\phi)$ are defined in 
\eqref{term1i}, \eqref{term3b}, \eqref{term2b}, \eqref{fluc1} and \eqref{fluc2} respectively; $Y^e_t(\phi)$ is defined in   \eqref{term2a}; the new term
\begin{align}
	& T^e_t(\phi):= \int_0^t   \alpha_e\, u^n_{s-}(x^e_2)\,\nabla_L\phi_s(x^e_1)\;ds \label{Te}
\end{align}
is obtained from \eqref{GlueT0}; note that the $\alpha_e\Delta_L \phi_s(x)$ term is killed due to our choice of $\phi$ in \eqref{E:testfcn}. New terms $\sum_eU^e_t(\phi)$ and $\sum_eV^e_t(\phi)$  are defined in  \eqref{aver1} and \eqref{aver2} respectively.

\medskip

{\bf New technical challenge in proving tightness:} The four terms \eqref{E:Green_n} are analogous to terms in (36) of \cite{MR3582808}, but  all five
terms in \eqref{E:Green_n2} and the $\partial_s\phi_s$ term in \eqref{E:Green_n0}  are new: they come from boundary terms at vertices of $\Gamma$. Treating these new terms  requires the uniform  estimates for the transition density $p^n(t,x,y)$ of random walks on graph, as well as the careful choice of $C^n_{x,y}$ in Condition \eqref{Cond_C}.

\medskip
{\bf Cancellation and simplification:}
An important observation is that,
by our choice of  $C^n_{x,y}$, equations \eqref{E:Green_n0}-\eqref{E:Green_n2} simplify to
\begin{align}
	u^n_t(y) =\,& P^n_{t}u^n_0(y) + \sum_e \Big(Y^e_t(\phi)+ Z^e_t(\phi)+E^{(1,e)}_t(\phi) + E^{(2,e)}_t(\phi)\Big) \label{E:Green_n_sim}\\
	&  + \sum_e \Big( E^{(3,e)}_t(\phi)+ E^{(4,e)}_t(\phi)\Big) +o(1). \label{E:Green_n2_sim}
\end{align}
To see this, note that  $U^e_t(\phi)=err_1(t)$ and $V^e_t(\phi)=err_2(t)$ defined in 
\eqref{err1} and \eqref{err2} respectively. That is,
\begin{align*}
	U^e_t(\phi)= &\int_0^t  \sum_{(e,\tilde{e})\text{unordered}}\,
	2C^n_{\tilde{e},e}\big(\phi_s(x^e_1)-\phi_s(x^{\tilde{e}}_1)\big)\,\big(u^n_{s-}(x^{\tilde{e}}_1) - u^n_{s-}(x^e_1)\big)\,ds \\
	V^e_t(\phi)= & \,-\int_0^t   \alpha_e\big(u^n_{s-}(x^e_2) - u^n_{s-}(x^e_1)\big)\,\nabla_L\phi_s(x^e_1)\;ds.
\end{align*}
On other hand, by \eqref{Generator_n2}
\begin{equation*}
	\mathcal{L}_{n}F(x^e_1)\approx\,\alpha_e(L^e)^2\big(F(x^{e}_2) -F(x^e_1)\big)\,+\,\sum_{\tilde{e}\in E(v):\,\tilde{e}\neq e}\big(F(x^{\tilde{e}}_1) -F(x^e_1)\big)\,C^n_{e,\tilde{e}}\,L^e.
\end{equation*}
where the approximation $\approx$ is quantified in the last sentence in Remark \ref{Rk:RWGen}.
So by our choice of $\phi$ and Komogorov's equation, 
\begin{equation}\label{partial_sphi}
	\partial_s\phi_s(x^e_1)\approx\,-\alpha_e(L^e)^2\big(\phi_s(x^{e}_2) -\phi_s(x^e_1)\big)\,-\,\sum_{\tilde{e}\in E(v):\,\tilde{e}\neq e}C^n_{e,\tilde{e}}L^e\big(\phi_s(x^{\tilde{e}}_1) -\phi_s(x^e_1)\big).
\end{equation}

From these it is easy (for example $V$ cancels with the first terms of \eqref{partial_sphi}, $T$ cancels with part of $U$) to check, by using symmetry $C^n_{e,\tilde{e}}=C^n_{\tilde{e},e}$, that we have cancellations
\begin{equation}\label{CancelG}
	0\approx\sum_e\frac{1}{L^e} \int_0^t u^n_s(x^e_1)\partial_s\phi_s(x^e_1)\,ds+  T^e_t(\phi)+U^e_t(\phi)+ V^e_t(\phi),
\end{equation}
giving the desired \eqref{E:Green_n_sim} and \eqref{E:Green_n2_sim}. 


\subsection{Tightness of approximate densities}\label{S:tight}

Our goal of this section is to prove the following $\mathcal{C}$-tightness result.
\begin{prop}\label{T:Tight}
	Suppose the assumptions in Theorem \ref{T} hold. Then the sequence $\{u^n\}_{n\geq 1}$ is tight in  
	$D([0,T],\,\mathcal{C}_{[0,1]}(\Gamma))$ for every $T>0$. Moreover,  any subsequential limit has a continuous version.
\end{prop}

Tightness in $D([0,T],\,\mathcal{M}(\Gamma))$, where $\mathcal{M}(\Gamma)$ is the space of positive Radon measures on $\Gamma$, is much easier to prove than tightness in $D([0,T],\,\mathcal{C}_{[0,1]}(\Gamma))$ since the former can be reduced to one-dimensional tightnesses (by integrating a fixed function against the measures). 
However, it is not easy to identify subsequential limits of measures as those whose density solves SPDE \eqref{fkpp1} weakly.

\begin{proof}[Proof of Proposition \ref{T:Tight}]
	The desired $\mathcal{C}$-tightness follows once we can show that (i) the ``weak" compact containment condition (condition (a) of Theorem 7.2 in \cite[Chapter 3]{MR838085}) holds and (ii) for any $\epsilon>0$, one has
	\begin{align}
		& \lim_{\delta\to 0} \limsup_{n\to\infty}\,\P\,\bigg( \sup_{\substack{t_1-t_2<\delta\\0\leq t_2\leq t_1\leq T}}
		\big\|u^n_{t_1}-u^n_{t_2} \big\|\,>\epsilon \bigg) = 0.
		\label{E:Tightu}
	\end{align}	
	Here and in what follows the norm is the one defined in \eqref{normC}. 
	It is enough to show that \eqref{E:Tightu} holds with $u^n$ replaced by any term in the decomposition given in \eqref{E:Green_n_sim}-\eqref{E:Green_n2_sim}.

	\mn
	{\bf First term in \eqref{E:Green_n_sim}. }
	Upon linearly interpolating  
	$ P^n_{t}u^n_0(x)$ in space, triangle inequality and the contraction property of the semigroup $\{P^n_t\}_{t\geq 0}$ gives
	\begin{equation}
		\label{E:Semigroups}
		\sup_{\substack{t_1-t_2<\delta\\0\leq t_2\leq t_1\leq T}}
		\big\|P^n_{t_1}u^n_0-P^n_{t_2}u^n_0 \big\|\leq 
		2\,\|u^n_0-f_0\| + \sup_{\substack{t_1-t_2<\delta\\0\leq t_2\leq t_1\leq T}}
		\big\|P^n_{t_1}f_0-P^n_{t_2}f_0 \big\|,
	\end{equation}
	where 
	$f_0$ is the initial condition for $u$ functions in Theorem  \ref{T}.
	So to show that \eqref{E:Tightu} holds with $u^n_t$ replaced by $P^n_{t}u^n_0$, it suffices to show that \eqref{E:Tightu} holds with $u^n_t$ replaced by $P^n_{t}f_0$. The latter can be checked by using 
	the uniform H\"older continuity \eqref{E:HolderCts2}.
	
	\mn
	{\bf Second term in \eqref{E:Green_n_sim}. } 
	For simplicity, we write 
	\begin{align*}
		\widehat{u}^n_t(y)&:=\sum_e \Big(Y^e_t(\phi)+ Z^e_t(\phi)+E^{(1,e)}_t(\phi) + E^{(2,e)}_t(\phi)\Big).
	\end{align*}
	The next moment estimate for space and time increments is similar to  \cite[Lemma 6]{MR1346264} and \cite[Lemma 4]{MR3582808}.
	\begin{lemma}
		\label{L:Moment_hatu}
		For any $p\geq 2$ and $T\geq 0$ and compact subset $K$ of $\Gamma$, there exists a constant $C_{T,p,K}\in(0,\infty)$  such that
		\begin{align}
			\label{E:Moment_hatu}
			&\E |\widehat{u}^n_{t_1}(y_1)- \widehat{u}^n_{t_2}(y_2)|^p \leq C_{T,p,K}\,\Big(|t_1-t_2|^{p/4}+d(y_1,y_2)^{p/2}+ (\inf_e M^e)^{-p} \Big) 
		\end{align}
		for all $0\leq t_2\leq t_1\leq T$, $y_1,\,y_2\in \Gamma^n\cap K$ and $n\geq 1$.
	\end{lemma}	
\begin{proof}[Proof of Lemma \ref{L:Moment_hatu}]
	The proof of this result  requires a number of computations involving estimations of the heat kernel $p^n(t,x,y)$.
	By \eqref{Y^e_tAlso},
	$Y^e_t(\phi)$ is equal to
	\begin{align}
	&  \frac{M^eB^n_e}{L^e} \sum_{x\in e^n} \int_0^t [u^n_{s-}(x-L^{-1}) ] (1-u^n_{s-}(x)) \phi_s(x) \, ds \label{Ye1}\\
	+\,& \frac{M^eB^n_e}{L^e} \sum_{x\in e^n\setminus \{x^e_1,x^e_2\}} \int_0^t [u^n_{s-}(x+L^{-1}) ] (1-u^n_{s-}(x)) \phi_s(x) \, ds \label{Ye2}\\
	+\,& \frac{M^eB_{e,e}}{L^e}  \int_0^t u^n_{s-}(x^e_1)  (1-u^n_{s-}(x^e_2)) \phi_s(x^e_2) \, ds \label{Ye3}\\
	+\,& \sum_{\tilde{e}\in E(v)\setminus e} \frac{M^{\tilde{e}}B_{e,\tilde{e}}}{L^e}  \int_0^t u^n_{s-}(x^{\tilde{e}}_1)  (1-u^n_{s-}(x^e_1)) \phi_s(x^e_1) \,  ds \label{Ye4}
	\end{align}
	To estimate $Y^e_{t_1}(\phi^{t_1,y_1})-Y^e_{t_2}(\phi^{t_2,y_2})$, the key is to observe that all four terms \eqref{Ye1}-\eqref{Ye4} are sum of integrals of the form
	$$I_t(y):=\int_0^tv^n_{s-}(x)\,\phi^{t,y}_s(x)\,ds$$
	where  $|v^n_{s-}(x)|\leq 1$ (for example, $v^n_{s-}(x)=[u^n_{s-}(x-L^{-1}) ] (1-u^n_{s-}(x))$ for \eqref{Ye1}), and that
	\begin{align*}
	|I_{t_1}(y_1)-I_{t_2}(y_2)|&=\Big|\int_{t_2}^{t_1}v^n_{s-}(x)p^n_{t_1-s}(x,y_1)ds + \int_0^{t_2} v^n_{s-}(x)[p^n_{t_1-s}(x,y_1)-p^n_{t_2-s}(x,y_2)]ds\Big|\\
	&\leq \int_{t_2}^{t_1}p^n_{t_1-s}(x,y_1)\,ds + \int_0^{t_2} \big|p^n_{t_1-s}(x,y_1)-p^n_{t_2-s}(x,y_2)\big|\,ds.
	\end{align*}
	Hence
	\begin{align*}
	&|Y^e_{t_1}(\phi^{t_1,y_1})-Y^e_{t_2}(\phi^{t_2,y_2})|\\
	\leq &\,  \frac{2M^eB^n_e}{L^e} \Big(\sum_{x\in e^n}\int_{t_2}^{t_1}p^n_{t_1-s}(x,y_1)\,ds + \int_0^{t_2} \big|p^n_{t_1-s}(x,y_1)-p^n_{t_2-s}(x,y_2)\big|\,ds\Big) \\
	&+ \frac{M^eB_{e,e}}{L^e}\Big( \int_{t_2}^{t_1}p^n_{t_1-s}(x^e_2,y_1)\,ds + \int_0^{t_2} \big|p^n_{t_1-s}(x^e_2,y_1)-p^n_{t_2-s}(x^e_2,y_2)\big|\,ds\Big)\\
	&+ \sum_{\tilde{e}\in E(v)\setminus e} \frac{M^{\tilde{e}}B_{e,\tilde{e}}}{L^e} \Big(\int_{t_2}^{t_1}p^n_{t_1-s}(x^e_1,y_1)\,ds + \int_0^{t_2} \big|p^n_{t_1-s}(x^e_1,y_1)-p^n_{t_2-s}(x^e_1,y_2)\big|\,ds\Big).
	\end{align*}
	The new error terms $E^{(3,e)}+E^{(4,e)}$
	\begin{align*}
	&\sum_e \Big( E^{(3,e)}_t(\phi)+ E^{(4,e)}_t(\phi)\Big)\\
	=&\sum_{e\in E(v)}(M^eL^e)^{-1}  \sum_{\tilde{e}\in E(v)\setminus e} \sum_{z\in x^{e}_1} \sum_{w\in x^{\tilde{e}}_1} \int_0^t (\xi_{s-}(w) - \xi_{s-}(z)) \phi_s(z) \,\big(dP^{z,w}_s-A^n_{e,\tilde{e}} \,ds\big)\\
	&+\sum_{e\in E(v)}(M^eL^e)^{-1}  \sum_{z\in x^e_2} \sum_{w \in x^e_1} \int_0^t (\xi_{s-}(w) - \xi_{s-}(z)) \phi_s(z) \,\big(dP^{z,w}_s- \,A^n_{e,e} \,ds\big) \notag\\
	& \qquad\qquad\qquad\qquad\qquad\quad+(\xi_{s-}(z) - \xi_{s-}(w)) \phi_s(w) \,\big(dP^{w,z}_s- A^n_{e,e} \,ds\big)
	\end{align*}
	can be treated in the same way, so do terms $Z^{e}$,  $E^{(1,e)}$ and $E^{(2,e)}$.
	The proof of Lemma \ref{L:Moment_hatu} can be completed as in \cite[Section 6]{MR3582808}, using the uniform estimates in Theorem \ref{T:RWHK}. 
\end{proof}
	
	Observe that the last term in \eqref{E:Moment_hatu} $\lim_{n\to\infty} (\inf_e M^e)^{-p}= 0$, by Conditions (a)-(b) and the assumption that $\inf_e \alpha_e>0$.
	It can then be shown, as in \cite[Section 5]{MR3582808} that \eqref{E:Moment_hatu} implies
	\eqref{E:Tightu} holds for $\widehat{u}^n$. 
	This idea is described in the paragraph before Lemma 7 in \cite{MR1346264} and page 648 of \cite{MR2786644}: we approximate the c\`adl\`ag process $\widehat{u}^n$ by  a {\it continuous} process $\tilde{u}$ and invoke a tightness criterion inspired by Kolmogorov's continuity theorem. 
	Finally, the ``weak" compact containment condition (condition (a) of Theorem 7.2 in \cite[Chapter 3]{MR838085})  follows from the fact $0\leq u^n\leq 1$, \eqref{E:Semigroups} and  Lemma \ref{L:Moment_hatu} with $t_1=t_2$.
	
	The proof of Proposition \ref{T:Tight} is complete. 	
\end{proof}

\begin{proof}[Proof of Theorem \ref{T}]
	The proof of Theorem \ref{T} is complete by Proposition \ref{T:Tight}, \eqref{subseqLimit} and Lemma \ref{WellposeFKPP}. 
\end{proof}

\section{Proof of Theorem \ref{prop:SDE}}\label{S:SDE}




The proof follows the footsteps of that of Theorem \ref{T}, and many estimates are simpler and very similar. So we provide only the key steps here.
As before we begin with the approximate martingale problems for $U^n$. 
Using this, we identify any sub-sequential limit and obtain a Green's function representation for  $U^n$. The proof is complete by establishing tightness, based on the  Green's function representation.

Before proving convergence, we first argue that $U^n$, after the described interpolation, is an element in $\mathcal{C}([0,\infty),\,\mathcal{C}_{[0,1]}(\Gamma))$. 
Indeed, if  the initial condition $U^n(0)$ is bounded between 0 and 1, then there exists a weak solution $\{U^n_x\}_{x\in \Gamma^n}$ to  \eqref{U1}-\eqref{U2} with initial condition $U^n(0)$ and takes values in [0,1] for all $t\geq 0$. This follows from an approximation argument  as for the SPDE  \eqref{fkpp0}:
one first approximates $\sqrt{x(1-x)}$ by Lipschitz functions $\{a_k(x)\}_{k}$ taking values in $[0,1]$ such that $a_k(0)=a_k(1)=0$ and constructs solutions $\{U^{n,k}\}_{k}$ using standard theory, then applies the comparison principle and passing to a sub-sequential limit $k\to\infty$. See, for instance, \cite[Section 2.1]{mueller2019speed}.
Furthermore,
weak uniqueness of \eqref{U1}-\eqref{U2}  follows from duality argument as in \cite{MR2014157}. Therefore $U^n(t)\in \mathcal{C}_{[0,1]}(\Gamma)$ for all $t\geq 0$. Finally, since $t\mapsto U^n_x(t)$ is continuous for each $x\in\Gamma^n$, and a compact subset of $\Gamma$ contains only finitely many $x\in \Gamma^n$, we have $U^n\in \mathcal{C}([0,\infty),\,\mathcal{C}_{[0,1]}(\Gamma))$ based on the norm of $\mathcal{C}_{[0,1]}(\Gamma)$ in \eqref{normC}.
 

\subsection{Approximate martingale problem}
This section is similar to Section \ref{S:ApproxMtg}.  Recall the inner products \eqref{Def:innerprod} and let $\phi: [0,\infty)\times \Gamma \to \RR$ be a continuous function that is continuously differentiable in $t$, twice continuously differentiable and has compact support in $x\in \Gamma$ and satisfies the gluing condition \eqref{glue_phi}.

As before we focus on a single vertex $v\in V$, and 
for each $e\in E(v)$ we enumerate the set $e^n$ as $(x_1^e,\,x_2^e,\,x_3^e,\cdots)$ along the direction of $e$ away from $v$; see {\bf Figure \ref{Fig:LatticeGn}}.
From our construction \eqref{U1}-\eqref{U2},  for $t\ge 0$ and  $e\in E(v)$, 
\begin{align}
&\<U^n(t),\,\phi_t\>_e-  \<U^N(0),\,\phi_0\>_e  -\int_0^t \<U^n(s),\partial_s\phi_s\>_e\,ds  
\nonumber\\
=\,& \frac{1}{L^e}\sum_{x\in e^n\setminus \{x^e_1\}}\left\{\int_0^t
 \phi_s(x)\Big[\mathcal{L}_n U_x + \beta_e\, U_x(1-U_x)\Big]  ds +
 \int_0^t \phi_s(x)
 \sqrt{\gamma_e\,L^e\,U_x(1-U_x)}  dB_x(s)\right\}
\label{term1_Un}\\
 &\,+\, \frac{1}{L^e}\int_0^t \phi^e_1\,
\left[\mathcal{L}_n U^e_1 + L^e\,\frac{\hat{\beta}(v)}{deg(v)}\, U^e_1(1-U^e_1) \right]  ds,
\label{term2_Un}
\end{align}
where we used the abbreviations $U^e_k:=U_{x_k^e}$ and  $\phi^e_k:=\phi_t(x_k^e)$. 

\smallskip
{\bf Laplacian and gluing condition. }  
We first consider the terms involving the operator $\mathcal{L}_n$ in  \eqref{term1_Un}- \eqref{term2_Un}. The careful construction of the generator $\mathcal{L}_n$ of the CTRW $X^n$ in Remark \ref{Rk:RW}, including condition \eqref{Cond_C} for the conductances $\{C_{e,\tilde{e}}\}$, as well as the gluing condition for $\phi$, are crucial for the lemma below. 
\begin{lemma}\label{L:Laplace_Un}
Suppose $\{U^n\}$ converge almost surely in $C([0,\infty),C_{[0,1](\Gamma)})$ along a sub-sequence to a limit $U^{\infty}$. Then along that sub-sequence, almost surely,
\[
\sum_{e\in E(v)}\frac{1}{L^e}\sum_{x\in e^n}  \phi_s(x)\,\mathcal{L}_n U^n_x(s) \to \sum_{e\in E(v)} \alpha_e \int_{e} U^{\infty}(s,x)\,\Delta \phi(s,x)\,m(dx).
\]
uniformly for all $t\in [0,\infty)$ and $v\in V$.
\end{lemma}

\begin{proof}
By Condition \eqref{Cond_C}, 
\begin{equation}\label{Def:theta_n}
\sup_{s\geq 0}\frac{1}{L^e}\sum_{x\in e^n\setminus \{x^e_1\}} \left|\mathcal{L}_n U^n_x(s) \,-\, \alpha_e \Delta_LU^n_{x}(s) \right| \to 0
\end{equation}
uniformly for all edges and vertices, where $\Delta_L$ is the discrete Laplacian defined in \eqref{dLaplacian}; see Remark \ref{Rk:RWGen}.
Applying summation by parts
twice and omitting the  superscript $e$ for simplicity,
\[\sum_{k=2}^N\phi_k(\nabla U_k-\nabla U_{k-1})=\phi_N \nabla U_N-\phi_2\nabla U_1-\nabla \phi_{N-1}U_N +\nabla \phi_2 U_2+\sum_{j=3}^{N-1}U_j(\nabla \phi_j -\nabla \phi_{j-1}),\]
where $\nabla U_k=U_{k+1}-U_k$. Since $\phi_s$ has compact support on $\Gamma$, 
\begin{align}\label{ByParts_0}
\sum_{k\geq 2}\phi^e_k\Delta_LU^{e}_{k} = 
- (L^e)^2\,[\phi^e_2\,(U^e_2-U^e_1)+(\phi^e_3-\phi^e_2)\,U^e_2]+\sum_{j\geq 3} U^e_j \Delta_L\phi^{e}_{j}
\end{align}
From the above two displays and also \eqref{Generator_n2},
we see that by
adding up the terms in \eqref{term1_Un}-\eqref{term2_Un} that involve the operator $\mathcal{L}_n$,  
\begin{align}
&\;\frac{1}{L^e}\sum_{x\in e^n}\phi_s(x)\,\mathcal{L}_n U_x \notag\\
=\,&\frac{\alpha_e}{L^e}\left\{
-(L^e)^2\,[\phi^e_2\,(U^e_2-U^e_1)+(\phi^e_3-\phi^e_2)\,U^e_2]+\sum_{j\geq 3} U^e_j \Delta_L\phi^{e}_{j} \right\}  \notag\\
&+\,\frac{1}{L^e}\,\phi^e_1\left\{ \alpha_e(L^e)^2\big(U^e_2 -U^e_1\big)\,+\,\sum_{\tilde{e}\in E(v):\,\tilde{e}\neq e}\big(U^{\tilde{e}}_1 -U^e_1\big)\,C_{e,\tilde{e}}\,L^e 
\right\}\,+\, \kappa_n\,\|\phi_s\|_{\infty} \notag\\ 
=\,&\alpha_e L^e\Big\{
- (\phi^e_3-\phi^e_2)\,U^e_2 -\big(\phi^e_2-\phi^e_1\big)\big(U^e_2 -U^e_1\big)
\Big\} \label{Un1}\\
&+\,\phi^e_1
\sum_{\tilde{e}\in E(v):\,\tilde{e}\neq e}\big(U^{\tilde{e}}_1 -U^e_1\big)\,C_{e,\tilde{e}}
+\frac{\alpha_e}{L^e}\sum_{j\geq 3} U^e_j \Delta_L\phi^{e}_{j} \,+\, \kappa_n\,\|\phi_s\|_{\infty} , \label{Un2}
\end{align}
where, by \eqref{Def:theta_n}, $\kappa_n$ is a constant that tends to 0 uniformly for all $e\in E(v)$, $v\in V$ and $s\geq 0$. 

We now sum over $e\in E(v)$. Suppose $\{U^n\}$ converge almost surely along a sub-sequence. Then by the \textit{gluing condition} \eqref{glue_phi}, the term \eqref{Un1} (after summing over $e\in E(v)$) will tend to zero along that sub-sequence. The first term of \eqref{Un2} also tends to zero (after summing over $e\in E(v)$), by the same argument for \eqref{aver1_result}-\eqref{err1} based on the symmetry of the conductances $\{C_{e,\tilde{e}}\}$ and the uniform continuity of $\phi$. The second term of \eqref{Un2}, converges to $\alpha_e \int_{e} U^{\infty}(s,x)\,\Delta \phi(s,x)\,m(dx)$. The proof is complete.
\end{proof}

\smallskip
{\bf Limiting martingale problem. } Convergences of the other terms in  \eqref{term1_Un}- \eqref{term2_Un} follow directly from the continuity of $\phi$. For example, for the term involving $\hat{\beta}(v)$,
\[
\sum_{e\in E(v)}\frac{1}{L^e}\int_0^t\phi^e_1\, L^e\,\frac{\hat{\beta}(v)}{deg(v)}\, U^e_1(1-U^e_1)  \, ds \to 
\hat{\beta}(v)\int_0^t\phi_s(v)\, U^{\infty}_v(s)\,\big(1-U^{\infty}_v(s)\big)\,ds,
\]
which explains the term  $deg(v)$  in the denominator.
Combining our calculations as in \eqref{LimitingMtg}-\eqref{LimitingQuad},  we have shown that, under convergence in distribution in 	$\mathcal{C}([0,\infty),\,\mathcal{C}_{[0,1]}(\Gamma))$,
\begin{equation}\label{subseqLimit_Un}
\text{ any sub-sequential limit of } \{U^n\}\text{ solves the SPDE } \eqref{fkpp1} \text{ weakly.}
\end{equation}

\subsection{Green's function representation}
This section is similar to Section \ref{S:green}. We shall obtain a Green's function representation of $U^n$ that is analogous to \eqref{E:Green_n_sim}-\eqref{E:Green_n2_sim}. Such representation will be useful for proving tightness.

Applying the approximate martingale problem \eqref{term1_Un}-\eqref{term2_Un}  with the same test function $\phi_s:=\phi^{t,y}_s$ in \eqref{E:testfcn}, namely
\begin{equation*}
\phi_s(x):=\phi^{t,y}_s(x):= 
 p^{n}(t-s,x,y)  \quad \text{for }s\in[0,t]
\end{equation*}
and zero otherwise, we obtain that  $\partial_s\phi_s+\mathcal{L}_n\phi_s =0$ and hence
for $t\geq 0$ and $y\in \Gamma^n$,
\begin{align}
&U^n_y(t) -  \big(P^n_{t}u_0\big)(y)  +\int_0^t \< U^n(s),\,\mathcal{L}_n\phi_s \> - \< \mathcal{L}_n U^n(s),\,\phi_s \> \,ds
 \label{green0_Un}\\
=\,&\sum_{e} \frac{1}{L^e}\sum_{x\in e^n\setminus \{x^e_1\}}\left\{\int_0^t
\phi_s(x)\,\beta_e\, U_x(1-U_x) \, ds +
\int_0^t \phi_s(x)
\sqrt{\gamma_e\,L^e\,U_x(1-U_x)} \, dB_x(s)\right\}
\label{green1_Un}\\
&\,+\, \sum_{e} \frac{1}{L^e}\int_0^t \phi^e_1\,
 L^e\,\frac{\hat{\beta}(v)}{deg(v)}\, U^e_1(1-U^e_1) \, ds
\label{green2_Un}.
\end{align}


\subsection{Tightness}
This section is similar to Section \ref{S:tight}. Our goal is to prove the following tightness result.
\begin{prop}\label{T:Tight_Un}
	Suppose the assumptions in Theorem \ref{prop:SDE} hold. Then the sequence $\{U^n\}_{n\geq 1}$ is tight in  
	$C([0,T],\,\mathcal{C}_{[0,1]}(\Gamma))$ for every $T>0$. 
\end{prop}

\begin{proof}
The desired tightness follows once we can show that (i) the ``weak" compact containment condition (condition (a) of Theorem 7.2 in \cite[Chapter 3]{MR838085}) holds and (ii) for any $\epsilon>0$, one has
	\begin{align}
	& \lim_{\delta\to 0} \limsup_{n\to\infty}\,\P\,\bigg( \sup_{\substack{t_1-t_2<\delta\\0\leq t_2\leq t_1\leq T}}
	\big\|U^n(t_1)-U^n(t_2) \big\|\,>\epsilon \bigg) = 0.
	\label{E:TightU_n}
	\end{align}		
To show \eqref{E:TightU_n}, it is enough to show that \eqref{E:TightU_n} holds with $U^n$ replaced by any other terms in the decomposition given in \eqref{green0_Un}-\eqref{green2_Un}.

Note that the summation-by-parts calculation \eqref{ByParts_0} still holds if $\phi$ and $U$ are interchanged. From this, the fact that $0\leq U^n\leq 1$ and the Gaussian upper bound \eqref{e:2.14}-\eqref{e:2.15} for $\phi_s$, we obtain that the integral in \eqref{green0_Un}  
\begin{align}\label{SymDifference}
\left|\int_0^t \< U^n(s),\,\mathcal{L}_n\phi_s \> - \< \mathcal{L}_n U^n(s),\,\phi_s \> \,ds\right|\leq \kappa_n\,\sqrt{t} 
\end{align}
for $t\geq 0$, where $\kappa_n$ is a constant that tends to 0 uniformly for all $e\in E(v)$, $v\in V$ and $t\geq 0$.

Furthermore, as in Lemma \ref{L:Moment_hatu}, if we let 
$\widehat{U}^n_y(t)= \eqref{green1_Un} \,+\,\eqref{green2_Un}$,
then we have the following moment estimate:
For any $p\geq 2$ and $T\geq 0$ and compact subset $K$ of $\Gamma$, there exists a constant  $C_{T,p,K}\in(0,\infty)$  such that
	\begin{align}
	\label{E:Moment_hatu_Un}
	&\E \left|\widehat{U}^n_{y_1}(t_1)- \widehat{U}^n_{y_2}(t_2)\right|^p \leq C_{T,p,K}\,\Big(|t_1-t_2|^{p/4}+d(y_1,y_2)^{p/2} \Big) 
	\end{align}
	for all $0\leq t_2\leq t_1\leq T$, $y_1,\,y_2\in \Gamma^n\cap K$ and $n\geq 1$.

The ``weak" compact containment condition (condition (a) of Theorem 7.2 in \cite[Chapter 3]{MR838085})  follows from the fact $0\leq U^n\leq 1$ and  \eqref{E:Moment_hatu_Un} with $t_1=t_2$. With \eqref{SymDifference},  \eqref{E:Moment_hatu_Un} and  the uniform heat kernel estimates in  Theorem \ref{T:RWHK}, the proof can be completed as in Section \ref{S:green}.

The proof of Proposition \ref{T:Tight_Un} is complete. 	
\end{proof}

\begin{proof}[Proof of Theorem \ref{prop:SDE}]
The proof of Theorem \ref{prop:SDE} is complete by Proposition \ref{T:Tight_Un}, \eqref{subseqLimit} and Lemma \ref{WellposeFKPP}. 
\end{proof}

\medskip

\section{Appendix}
In this appendix, we specify the  notions of weak solution and mild solution to SPDE on graphs, give a proof of Lemma \ref{L:WFdual}, point out some generalizations of our results and some open problems.

\subsection{Solutions to SPDE on graphs}

For completeness we give the precise definition for the notions of weak solutions and of mild solutions for SPDE \eqref{fkpp0}:
\begin{equation*}
\left\{\begin{aligned}
\partial_t u      &\,= \mathcal{L} \,u +\beta\,u(1-u) + \sqrt{\gamma\,u(1-u)}\dot{W}
& &\quad\text{on }\overset{\circ}{\Gamma} \\
\nabla_{out}u\cdot [\alpha] &\,= -\hat{\beta}\, u(1-u)  & &\quad \text{on }V.
\end{aligned}\right.
\end{equation*}

These definitions are analogous to the usual ones (see \cite{MR876085}) but have extra boundary conditions.

\begin{definition}
	Let  $\dot{W}$ be the space-time white noise on $\Gamma$ endowed with the product of Lebesque measures $\mu:=m(dx)\otimes dt$. That is, 
	$\dot{W}:= \{\dot{W}(A)\}_{A\in \mathfrak{B}(\Gamma\times[0,\infty))}$ are centered Gaussian random variables with covariance 
	$\E[\dot{W}(A)\,\dot{W}(B)]=\mu(A\cap B)$.
	The white noise process $(W_t)_{t\geq 0}$ is defined by $W_t(U):=\dot{W}(U\times[0,t])$ where $U\in \mathfrak{B}(\Gamma)$. Denote by $\mathcal{F}_t$ to the sigma-algebra generated by $\{W_s(U):\;0\leq s\leq t,\,U\in \mathfrak{B}(\Gamma)\}$ and call $(\mathcal{F}_t)_{t\geq 0}$ the filtration generated by $\dot{W}$.
\end{definition}
It can be checked as in \cite{MR876085} that $\{W_t(U),\,\mathcal{F}_t\}_{t\geq 0,\,U\in \mathfrak{B}(\Gamma)}$ is an orthogonal (hence worthy) martingale measure,
so that we have a well-defined notion of stochastic integral with respect to $W$ for a class of integrands which contains  the collection of all predictable functions $f$ such that 
\begin{equation*}
	\E\int_{(0,T]}\Big(\int_\Gamma|f(x,t)|\,m(dx)\Big)^2\,dt<\infty 
\end{equation*}
for all $T>0$. By the Gaussian upper bound in \eqref{DiffusionGE}, the stochastic integrals that appear throughout this paper, including  \eqref{E:MildSol} below, are well-defined.

The notions of  weak solution and  mild solution are given below. 
Let 
$\mathcal{B}_{[0,1]}(\Gamma)$ be the space of Borel measurable functions on $\Gamma$ taking values in the interval $[0,1]$. Note that we do not need an absolute sign inside the square root  $\sqrt{\gamma u(1-u)}$.
\begin{definition}\label{Def:WeakSol}
	A process  $u=(u_t)_{t\geq 0}$ taking values in $\mathcal{B}_{[0,1]}(\Gamma)$  is a {\bf weak solution} to SPDE \eqref{fkpp0} with initial condition $u_0$ if there is a space-time white noise $\dot{W}$ on $\Gamma\times[0,\infty)$ such that (i) $u$ is adapted to the filtration generated by $\dot{W}$ and (ii) for any $\phi\in \mathcal{C}_c(\Gamma)\cap \mathcal{C}^2(\mathring{\Gamma})$ satisfying the gluing condition \eqref{glue},
	we have
	\begin{align}\label{E:WeakSol}
		\int_\Gamma u_t(x)\phi(x)\,\nu(dx) &= \int_\Gamma u_0(x)\phi(x)\,\nu(dx)
		+\frac{1}{2}\int_0^t\int_{\Gamma} \nabla\Big(\ell\,\alpha\,\nabla \phi\Big)(x)\,u_s(x)\,m(dx)\,ds \notag\\	
		&+ \int_0^t\int_\Gamma \beta(x)\,u_s(x)(1-u_s(x))\,\phi(x)\,\nu(dx)\,ds \notag\\
		&+ \int_{[0,t]\times \Gamma}
		\phi(x)\sqrt{\gamma(x)\,u_s(x)\big(1-u_s(x)\big)}\,\ell(x)\,dW(s,x) \notag\\
		& +\frac{1}{2}\int_0^t\sum_{v\in V}\phi(v)\hat{\beta}(v)\,u_s(v)(1-u_s(v))\,\ell(v)\,ds
	\end{align}
	for all $t\geq 0$, almost surely, where $\nu(dx)=\ell(x)m(dx)$. 
\end{definition}

\begin{definition}\label{Def:MildSol}
A process  $u=(u_t)_{t\geq 0}$ taking values in $\mathcal{B}_{[0,1]}(\Gamma)$ is a {\bf mild solution} to SPDE \eqref{fkpp0} with initial condition $u_0$ if there is a space-time white noise $\dot{W}$ on $\Gamma\times[0,\infty)$ such that (i) $u$ is adapted to the filtration generated by $\dot{W}$ and (ii)
$u$ solves the integral equation
\begin{align}\label{E:MildSol}
u_t(x)= P_t u_0(x) &+ \int_0^t P_{t-s}\big(\beta\,u_s(1-u_s)\big)(x)\,ds \notag \notag\\
&+ \int_{[0,t]\times \Gamma}p(t-s,x,y)\,\ell(y)
\sqrt{\gamma(y)\,u_s(y)\big(1-u_s(y)\big)}\,dW(s,y)\notag\\
&+\frac{1}{2}\int_0^t\sum_{v\in V}p(t-s,x,v)\,\ell(v)\,\hat{\beta}(v)\,u_s(v)(1-u_s(v))\,ds,
\end{align}
where $p(t,x,y)$ is defined in \eqref{Def:p} and $P_tf(x)=\E_xf(X_t)=\int_{\Gamma}f(y)p(t,x,y)\nu(dy)$.
\end{definition}

A weak solution may fail to be a mild solution in general (see, e.g., \cite{MR2981850}). However, we have the following.
\begin{lemma}\label{L:weakmild}
Suppose Assumptions \ref{A:graph}, \ref{A:Coe_a} and \ref{A:Coe_b} hold and
 $u_0\in \mathcal{C}_{[0,1]}(\Gamma)$. An adapted process
$u$ with  $u_t\in \mathcal{B}_{[0,1]}(\Gamma)$ for all $t\geq 0$ is a weak solution of \eqref{fkpp0} if and only if it is a mild solution. 
\end{lemma}

\begin{proof}
A proof follows from that of Shiga's result \cite[Theorem 2.1]{MR1271224}, thanks to our heat kernel estimates in Theorem \ref{T:DiffusionHK} and the boundedness of $u$. We provide a sketch of proof.	

Suppose $u$ is a weak solution of \eqref{fkpp0}. Then \eqref{E:WeakSol} holds for the test function $\phi(x)= p(t, y,x)$ for $x\in \Gamma$, where $(t,y)\in (0,\infty)\times \Gamma$ is fixed. This is because (i) $\phi\in Dom_{L^2}(\mathcal{L})$
and hence $\phi$ satisfies the gluing condition, and (ii) 
 $Dom_{L^2}(\mathcal{L}) \subset Dom (\mathcal{E})=W^{1,2}(\Gamma,\nu)\cap \mathcal{C}^1(\Gamma)$ and so $\phi$ is the limit, under the Sobolev norm $\|\cdot \|_{W^{1,2}(\nu)}$, of a sequence in   $\mathcal{C}_c(\Gamma)\cap \mathcal{C}^2(\mathring{\Gamma})$ that satisfy the gluing condition.
Furthermore, the time-derivatives of $p(t,x,y)$ also possess Gaussian bounds and H\"older continuity; see
\cite[Theorem 2.32]{gyrya2011neumann}. This allows us to extend  \eqref{E:WeakSol} to time-dependent test functions $\phi:\,[0,T]\times \Gamma \to \RR$, where $\phi(t,x)=p(T-t,y,x)$ and $T,y$ are fixed. This gives 
\begin{align}\label{E:MildSol_t}
P_{T-t}u_t(x)= P_T u_0(x) &+ \int_0^t P_{T-s}\big(\beta\,u_s(1-u_s)\big)(x)\,ds \notag \notag\\
&+ \int_{[0,t]\times \Gamma}p(T-s,x,y)\,\ell(y)
\sqrt{\gamma(y)\,u_s(y)\big(1-u_s(y)\big)}\,dW(s,y)\notag\\
&+\frac{1}{2}\int_0^t\sum_{v\in V}p(T-s,x,v)\,\ell(v)\,\hat{\beta}(v)\,u_s(v)(1-u_s(v))\,ds
\end{align}
and hence \eqref{E:MildSol} by letting $t\to T$. Therefore $u$ is also a mild solution.

The converse is straightforward.
Suppose $u$ is a mild solution. We insert \eqref{E:MildSol} into $\int_0^t\int_{\Gamma} \nabla\Big(\ell\,\alpha\,\nabla \phi\Big)(x)\,u_s(x)\,m(dx)\,ds$ and apply a stochastic Fubini theorem to verify 
\eqref{E:WeakSol}.
\end{proof}

For existence and uniqueness of solutions,
results  in \cite{MR3634278} cannot be directly applied since  we have a non-Lipschitz coefficient $\sqrt{\gamma\,u(1-u)}$.
Nonetheless,  existence of a mild solution taking values in $\mathcal{B}_{[0,1]}(\Gamma)$ follows from an approximation and tightness argument.
See, for instance, \cite[Section 2.1]{mueller2019speed}. Furthermore, weak uniqueness for \eqref{fkpp0} holds by Lemma \ref{WellposeFKPP}.

\begin{proof}[Proof of Lemma \ref{L:WFdual} (Duality)]
Recall that  Assumptions \ref{A:graph}, \ref{A:Coe_a} and \ref{A:Coe_b} in force. Recall from \eqref{Def:nu} and \eqref{Def:p} that $\nu(dy)= \ell(y) m(dy)$ and $\P_x(X_t\in dy)=p(t,x,y)\,\nu(dy)$.

To simplify notation, we write $p^{\epsilon}(x,y)=p(\ep,x,y)$ for the transition density of the symmetric diffusion $X$ and
define $z:=1-u$ and $\bar z_t(x):=\int_\Gamma z_t(y)p^{\ep}(x,y)\nu(dy)$. 

Using the weak formulation \eqref{E:WeakSol} with test function $\phi^{\ep,x}(y)= p^{\ep}(x,y)$ (we can do so as explained in the proof of Lemma \ref{L:weakmild}), we obtain
\begin{align*}
\bar z_t(x) - \bar z_0(x)&=
\frac{1}{2}\int_0^t\int_{\Gamma} \nabla_y\Big(\ell(y)\,\alpha(y)\,\nabla_y p^{\ep}(x,y)\Big)\,\,z_s(y)\,m(dy)\,ds \notag\\	
&- \int_0^t\int_\Gamma \beta(y)\,z_s(y)(1-z_s(y))\,p^{\ep}(x,y)\,\ell(y)\,m(dy)\,ds \notag\\
&+ \int_0^t\int_\Gamma
p^{\ep}(x,y)\sqrt{\gamma(y)\,z_s(y)\big(1-z_s(y)\big)}\,\ell(y)\,dW(s,y)\\
&+\frac{1}{2}\int_0^t\sum_{v\in V}p^{\ep}(x,v)\,\hat\beta(v)z_s(v)\big(1-z_s(v)\big)\,\ell(v)\,ds.
\end{align*}

Fix $x_1, \ldots x_n$ and let $\mathcal{L}_{\bar z}$ be the generator  of the process $(\bar z_t(x_1), \ldots, \bar z_t(x_n))_{t\geq 0}$. 
Using the It\^o's formula \cite[Lemma 2.3]{MR1743769} (each process $(\bar z_t(x))_{t\geq 0}$ is a semi-martingale so this is legitimate), we see that as in  Section 8.1 of \cite{MR3582808},
\begin{align}
\hbox{drift}\left( \mathcal{L}_{\bar z} \prod_i \bar z_t(x_i) \right) 
&= 
\frac{1}{2}\sum_{i=1}^n \prod_{j \neq i} \bar z_t(x_j) \int_\Gamma\nabla_y\Big(\alpha(y)\,\nabla_y p^{\ep}(x_i,y)\Big)\,\,z_t(y)\,m(dy)
\nonumber\\
&+  \sum_{i=1}^n \prod_{j \neq i} \bar z_t(x_j) \int_\Gamma \beta(y)[z_t^2(y) - z_t(y)] p^{\ep}(x_i,y) \,\ell(y)\,m(dy) 
\notag\\
& +  \sum_{i=1}^{n-1} \sum_{j=i+1}^n \prod_{k \neq i,j} \bar z_t(x_k) \int_\Gamma \gamma(y)[ z_t(y)(1 - z_t(y)) ] p^{\ep}(x_i,y) p^{\ep}(x_j,y) \ell(y) m(dy)
\nonumber\\
&+\frac{1}{2}\sum_{i=1}^n\sum_{v\in V}\prod_{j\neq i}\bar z_t(x_j)\,\,\hat\beta(v)\big(z_t^2(v)-z_t(v)\big)\,p^{\ep}(x_i,v)\,\ell(v) \label{udual}.
\end{align}
The dual process is a system of branching coalescing particles performing $\mathcal{E}$-diffusions on $\Gamma$. During their lifetime all particles  independently give birth (here means the particle splits into two) at rate $\beta(x)$ at the interior $x\in\mathring{\Gamma}$ and give birth at rate $\hat \beta(v)L^{i,v}_t$ at vertex $v\in V$.  
In addition, for $i < j$, particle $j$ is killed by particle $i$ at rate $\frac{\gamma(x_i)}{\ell(x_i)} L^{i,j}_t$ where 
$L^{i,j}_t$ denotes the local time of the process $x_j-x_i$ at $0$.
Writing $\mathcal{L}_x$ for the generator of this particle system, 
\begin{align}
\hbox{drift}\left(\mathcal{L}_x \prod_i \bar z_t(x_i) \right) &= 
\frac{1}{2}\sum_{i=1}^n \prod_{j \neq i} \bar z_t(x_j) \cdot\nabla\Big(\alpha\,\nabla \bar z_t\Big)(x_i)
\nonumber\\
&+ \sum_{i=1}^n  \prod_{j \neq i} \bar z_t(x_j) \cdot \beta(x_i)[\bar z_t^2(x_i)- \bar z_t(x_i) ] 
\notag \\
&+  \sum_{i=1}^{n-1} \sum_{j= i+1}^n  \prod_{k \neq i,j} \bar z_t(x_k) \cdot \left[\frac{\gamma(x_i)}{\ell(x_i)}\bar z_t(x_i)(1- \bar z_t(x_j)) \right]\,\delta_{\{x_j=x_i\}} \nonumber\\
&+ \frac{1}{2}\sum_{i=1}^n \sum_{v\in V} \prod_{j \neq i} \bar z_t(x_j) \cdot \hat\beta(v)[\bar z_t^2(v)- \bar z_t(v) ] \delta_{\{v=x_i\}}
\label{udual2}
\end{align}
in which we used the formal notation $dL^{i,j}_t=\delta_{\{x_j(t)=x_i(t)\}}\,dt$ and $dL^{i,v}_t= \delta_{\{x_i(t)=v\}}\,dt$.

We continue to follow \cite[Section 8.1]{MR3582808} to assert that
\begin{equation}\label{EKdual}
\E F(\bar z_t, x(0)) - \E F(\bar z_0, x(t)) = \E \int_0^t G(\bar z_{t-s}, x(s)) - H(\bar z_{t-s}, x(s)) \, ds.
\end{equation}
with  $G(\bar z, x)=\mathcal{L}_{\bar z}F(\bar z, x)$ and $H(\bar z, x)=\mathcal{L}_{x}F(\bar z, x)$, where
$$
F(\bar z, x) = \prod_{i=1}^{n} \bar z(x_i) \quad\text{if }x=(x_1,\cdots,\,x_n).
$$
To prove the desired duality formula \eqref{WFdual}, i.e.
\begin{equation}\label{WFdual'}
\E \prod_{i=1}^{n(0)} z_t(x_i(0)) =  \E \prod_{i=1}^{n(t)} z_0(x_i(t)),\quad t\geq 0,
\end{equation}
it remains to argue that the RHS of \eqref{EKdual} tends to zero as $\epsilon\to 0$.

The first term in  \eqref{udual} agrees with that of  \eqref{udual2} on the RHS of \eqref{EKdual}. The second and the third terms of  \eqref{udual} and  \eqref{udual2} can be teated in the same way as in \cite[Section 8.1]{MR3582808}, thanks to Theorem \ref{T:DiffusionHK} and Assumption  \ref{A:Coe_b}. The reason we need to divide $\gamma(x_i)$ by $\ell(x_i)$ in the third term of 
\eqref{udual2} is as follows: there are two $p^{\ep}$ factors but only one $\ell$ in the third term of \eqref{udual}.
This explains why the coalescence rate should be $\frac{\gamma(x_i)}{\ell(x_i)} L^{i,j}_t$ rather than  $\gamma(x_i) L^{i,j}_t$ as in \cite[Section 8.1]{MR3582808}.

Finally, we consider the new contribution to the RHS of \eqref{EKdual} from the last term of \eqref{udual} and  \eqref{udual2}, coming from the branchings in the vertex set $V$.

\bigskip

 The contribution to \eqref{EKdual} from the last term of \eqref{udual2} is
\begin{equation}\label{3rdLT0}
\E \int_0^t  
\sum_{i=1}^{n(s)} \sum_{v\in V} \prod_{j \neq i} \bar z_{t-s}(x_j(s)) \cdot \hat\beta(v)[\bar z_{t-s}^2(v)- \bar z_{t-s}(v) ] \,dL^{i,v}_s
\end{equation}
which converges as $\ep\to 0$ (first by dominated convergence  for fixed $v\in V$ and then by monotone convergence for $\sum_{v\in V}$) to
\begin{align}\label{3rdLT}
\E\int_0^t\sum_{i=1}^{n(s)} \sum_{v\in V} Y_{s,t}^{i,v} \,dL^{i,v}_s,
\end{align}
where 
$$
Y^{i,v}_{s,t}= \prod_{j \neq i}  z_{t-s}(x_j(s)) \cdot \hat\beta(v)[z_{t-s}^2(v)-  z_{t-s}(v) ]
$$
is non-negative and bounded from above uniformly for $1\leq i\leq n(s)$, $s\in[0,t]$ and $v\in V$,  according to Assumption \ref{A:Coe_b}. 
The contribution to \eqref{EKdual} from the last term of \eqref{udual} is
\begin{align}
\frac{1}{2}\E \int_0^t \sum_{i=1}^{n(s)}\sum_{v\in V}\,Y^{i,v}_{s,t}\,p^{\ep}(x_i(s),v)\,\ell(v)\,ds
\end{align}
which also converges to \eqref{3rdLT}  as $\ep\to 0$,  first by dominated convergence  for fixed $v\in V$ and then by monotone convergence for $\sum_{v\in V}$ (details as in  pages 1724--1725 of \cite{MR1813840}).
The proof of \eqref{WFdual'} is complete.
\end{proof}

\subsection{Generalizations}\label{S:Generalize}

	Assumptions \ref{A:graph} and \ref{A:Coe_a} can both be significantly relaxed.
	The Dirichlet form method enables one to construct more general processes, such as singular diffusions on graphs as in \cite{MR2563669, MR2906858}, and  diffusions on $\RR$-trees \cite{MR2469332, MR2357676, MR2675115, MR3034461} and  fractals \cite{MR966175}. These two assumptions should be compared with conditions on the boundary of a domain and on the coefficients respectively in the construction of reflected diffusions \cite{MR2849840}.	
	The  function $\ell$ contributes to a drift 
	for the diffusion which we ignored for simplicity, i.e., we assumed $\ell=1$. To incorporate such as drift in discrete approximations, see \cite[Section 6]{MR3485385}.	

	Furthermore, with the techniques developed in this paper, one can immediately generalize the coupled SPDE in \cite[Theorem 4]{MR3582808} to the graph setting. This enables us to investigate the role of space in shaping coexistence and competition outcomes of  interacting species. In ongoing work, we plan to apply such generalizations to study the interactions of virus and sub-virus particles during co-infection in a population of susceptible cells.

\subsection{Open problems}\label{S:Open}

The study of SPDE on graphs leads to an explosion of interesting open questions. We mention only three of the important ones below.

\begin{problem}[Speed of FKPP on $\Gamma$]\rm \label{SpeedFKPP}
	Wavefront propagation and asymptotic speed for  deterministic FKPP on a class of regular trees has recently been obtained using large deviation techniques in \cite{fan2019wave}. Can we define and obtain the {\it asymptotic speed} of wavefront propagation for the weak solution of the stochastic PDE \eqref{fkpp1} in terms of properties of $\Gamma$ ? For instance, in terms of $k$ (and $\alpha,\beta,\gamma$) when $\Gamma$ is the infinite $k$-regular tree with unit branch length? The duality formula in Lemma \ref{L:WFdual} can be very  useful for this study (see \cite{brunet2006phenomenological, MR2014157, MR2793860} for the case $\Gamma=\RR$). 
	
\end{problem}

\begin{problem}[Contact process]\rm \label{Otherspde}
	Techniques developed here enables one to analyse other SPDE on graphs such as scaling limits of
	branching random walks
	and those of {\it contact processes}
	\begin{equation}\label{contact}
	\partial_t u      \,= \alpha\,\Delta u +\beta\,u-\delta\,u^2 + \sqrt{\gamma \,u}\,\dot{W}
	\quad\text{on }\overset{\circ}{\Gamma}.
	\end{equation}
	A fundamental and challenging question in epidemiology is to
	estimate (and show existence of) the {\it threshold infection rate} $\beta_c \in (0,\infty)$ in terms of 
	geometric properties of $\Gamma$, such that when $\beta>\beta_c$ the infection sustain with positive probability, and when  $0\leq \beta<\beta_c$ the infection dies out with probability one
	(assuming $\beta$ is a constant function and $\hat{\beta}=0$ on $V$ for simplicity).
	See Muller and Tribe \cite[Theorem 1]{MR1296425} for such a threshold when $\Gamma=\RR$.
\end{problem}

\begin{problem}[Other SPDE and mixture models]\rm \label{Boundary}
	In this paper, we obtain the term $\hat{\beta}u(1-u)$ in \eqref{fkpp1} because we specify biased voter rule not just on edges but also at the vertex set. 
	Some other polynomial terms can arise if we specify other microscopic rules (such as contact process) near the vertices. This is left open for exploration. For example, the interacting particle system can be a mixture of the biased voter model and the contact process. See Lanchier and Neuhauser \cite{MR2209349, MR2739344} who introduced one such mixture to investigate how the interactions in spatially explicit host-symbiont systems are shaping plant community structure.
\end{problem}


\clearp 

\bibliographystyle{abbrv}
\bibliography{spdeG_S2}

\end{document}